\documentclass[12pt]{article}

\usepackage[T1,T2A]{fontenc}
\usepackage[utf8]{inputenc}

\usepackage{amsfonts}\usepackage{amssymb}
\usepackage{amsmath}
\usepackage{amsthm}
\usepackage{hyperref}
\usepackage{tabularx} 
\usepackage{mathtools, nccmath}
\usepackage{stackengine}
\usepackage{enumitem}

\usepackage{titling}

\stackMath

\newtheorem{theorem}{Theorem}[section]
\newtheorem{proposition}[theorem]{Proposition}
\newtheorem*{example}{Example}
\newtheorem{lemma}[theorem]{Lemma}
\newtheorem{obs}[theorem]{Observation}
\newtheorem{corollary}[theorem]{Corollary}
\newtheorem{definition}[theorem]{Definition}
\newtheorem{remark}[theorem]{Remark}
\newtheorem{thm}{Theorem} 
\newtheorem{lem}[thm]{Lemma} 
 
\newtheorem{prop}[thm]{Proposition} 

\newcommand{\alf}{{\underline{\alpha}}}
\newcommand{\betf}{{\underline{\beta}}}
\newcommand{\uf}{f}
\newcommand{\ug}{{g}}

\newcommand{\bfb}{{\bf b}}
\newcommand{\p}{{\bf p}}
\newcommand{\q}{{\bf q}}

\newcommand{\y}{{\bf y}}
\newcommand{\x}{{\bf x}}

\newcommand{\rrho}{{\bf \rho}}
\newcommand{\ssigma}{{\bf \sigma}}

\makeatletter
\newcommand{\leqnomode}{\tagsleft@true\let\veqno\@@leqno}%
\newcommand{\reqnomode}{\tagsleft@false\let\veqno\@@eqno}%
\newcommand*{\compress}{\@minipagetrue}
\makeatother

\newtheorem*{claim}{Claim}

\newenvironment{indentedclaim}
  {\begin{quote}\begin{claim}}
  {\end{claim}\end{quote}}

\usepackage{tikz} 
\usetikzlibrary{calc,arrows,decorations.pathreplacing,fadings,3d,positioning}

\makeatletter
\def\namedlabel#1#2{\begingroup
    #2%
    \def\@currentlabel{#2}%
    \phantomsection\label{#1}\endgroup
}
\makeatother

\title{Khintchine-type theorems for weighted uniform inhomogeneous approximations via transference principle}

\author{Vasiliy Neckrasov \\
        { \small Brandeis University,
        Waltham, MA 02453, USA} \\
        {\small \texttt{vneckrasov@brandeis.edu}}}

\date{}

\begin{document}

\maketitle

\begin{abstract}
    In \cite{kwad} Kleinbock and Wadleigh proved a "zero-one law" for uniform inhomogeneous Diophantine approximations. We generalize this statement with arbitrary weight functions and establish a new and simple proof of this statement, based on transference principle. We also give a complete description of the sets of $g$-Dirichlet pairs with a fixed matrix in this setup from Lebesgue measure point of view. As an application, we consider the set of badly approximable matrices and give a characterization of bad approximability in terms of inhomogeneous approximations. All the aforementioned metrical descriptions work (and sometimes can be strengthened) for weighted Diophantine approximations.
\end{abstract}

\vskip+0.5cm

\section{Introduction and main resluts}

{\it Throughout the text we mark our original statements by numbers (Lemma 1.1, Theorem 1.2, Proposition 2.3) and previously known statements by letters (Lemma A, Theorem B, Proposition C).}

\vskip+0.3cm

\subsection{The classical setup: Diophantine approximations of systems of affine forms}

Let $m$ and $n$ be positive integers and $d = m+n$. Let ${\bf M}_{n,m}$ be  the set of  all real $n \times m$ matrices. We will use the notation $| \cdot |$ for the supremum norm of a vector, and $|| \cdot ||$ for the distance to the nearest integer vector function. We will denote the $r$-dimensional Lebesgue measure of a set $E \subseteq \mathbb{R}^r$ by $m(E)$; the choice of dimension $r$ will be clear out of context, so we do not include it in the notation.

\vskip+0.2cm

The theory of uniform Diophantine approximation studies how small the value $|| \Theta \q ||$ can be for a given matrix $\Theta \in {\bf M}_{n,m}$ under certain restrictions on the size of the integer vector $\q$. The starting point of this area is the classical Dirichlet's theorem, which states that for any $\Theta \in {\bf M}_{n,m}$ and any real $T > 1$ the system of inequalities 

    $$
    \begin{cases}
    ||\Theta {\bf q}||^n \leq \frac{1}{T} \\
    |{\bf q}|^m \leq T
\end{cases}
    $$

has an integer solution $\q \in \mathbb{Z}^m$. More generally, one can replace the approximating function $f_1(T) = \frac{1}{T}$ (we will use the convention $f_a(T): = \frac{1}{T^a}$) by another non-increasing function $g: \,\,\, \mathbb{R}_+ \rightarrow \mathbb{R}_+$ and ask the same question about the system of inequalities 

    \begin{equation}\label{Dirichlet_hom}
    \begin{cases}
    ||\Theta {\bf q}||^n \leq g(T), \\
    |{\bf q}|^m \leq T.
\end{cases}
    \end{equation}

We call the matrix $\Theta$ {\it $g$-Dirichlet} if the system of inequalities \eqref{Dirichlet_hom} has a solution $\q \in \mathbb{Z}^m \setminus \{ \bf 0 \}$ for any $T \in \mathbb{R}_+$ {\it large enough}.

One can treat the quantity $\Theta \q$ as the vector of values that the system of linear forms 

$$
y_i = \theta_{i,1} x_1 + \ldots + \theta_{i,m} x_m, \,\,\,\,\,\,\,\, i = 1, \ldots, n
$$

attains on an integral $m$-tuple $\q$. A natural generalization for this setup is to consider affine forms 

$$
y_i = \theta_{i,1} x_1 + \ldots + \theta_{i,m} x_m - \eta_i, \,\,\,\,\,\,\,\, i = 1, \ldots, n
$$

instead. This requires the consideration of pairs $(\Theta, \pmb{\eta}) \in {\bf M}_{n,m} \times \mathbb{R}^n$ instead of matrices and motivates the following definition: we say that the pair $( \Theta, \pmb{\eta})$ is $g$-Dirichlet if the system of inequalities
    \begin{equation}
    \begin{cases}
        || \Theta  \q - \pmb{\eta} ||^n \leq g(T), \\
        |\q|^m \leq T
    \end{cases}
    \end{equation}
    
has a solution $\q \in \mathbb{Z}^m  \setminus \{ \bf 0 \}$ for any $T \in \mathbb{R}_+$ { large enough}.

Unlike the homogeneous case, not all the pairs are $f_1$-Dirichlet: moreover, it is not hard to see that in general the quantity $|| \Theta  q - \pmb{\eta} ||$ can even be bounded from below by a positive number. One can still, however, show that {\it Lebesgue almost all} pairs behave similarly to matrices and are $f_1$-Dirichlet. Moreover, in \cite{kwad}, Kleinbock and Wadleigh completely described the conditions under which the set of $g$-Dirichlet pairs has zero or full Lebesgue measure.

\begin{thm}\label{kwadresult}
    Let $g(T), \,\,\, T \geq 1$, be a non-increasing real valued function. The set of $g$-Dirichlet pairs has zero Lebesgue measure if the series

    $$
    \sum\limits_{l=1}^{\infty} \frac{1}{l^2 g(l)}
    $$

    diverges, and full measure if it converges.
\end{thm}

A stronger version of this theorem, providing a zero-one law for Hausdorff measures, was later proven by Kim and Kim in \cite{kimkim}.

\vskip+0.3cm

\subsection{Approximations with general weight functions: definitions}\label{general def section}

We will now introduce a more general setup of Diophantine approximations with respect to general weight functions.

Let 

\begin{equation}\label{alphbetdef}
\alf = (\alpha_1, \ldots, \alpha_n) \,\,\,\,\,\,\,\,\text{and} \,\,\,\,\,\,\,\, \betf = ( \beta_1, \ldots, \beta_m)
\end{equation}

be the $n$-tuple and $m$-tuple of positive, increasing, continuous real valued functions, defined on the interval $[1, +\infty)$ and having infinite limit an infinity. It will be convenient to assume that $\alpha_i(1) = \beta_j(1) = 1$. One can extend these functions to the whole set of nonnegative real arguments by defining
\begin{equation}\label{star_condition}
    \alpha_i(T) = \frac{1}{\alpha_i \left( \frac{1}{T} \right)} \,\,\,\,\,\,\,\,\,\,\,\, \text{and} \,\,\,\,\,\,\,\,\,\,\,\, \beta_j(T) = \frac{1}{\beta_j \left( \frac{1}{T} \right)}
\end{equation}

for $0 < T \leq 1$, and, by continuity, $\alpha_i(0) = \beta_j(0) = 0$. We will refer to these functions as {\it weight functions}; throughout the paper, the functions $\alpha_i$ and $\beta_j$ will always be defined as above, and in particular satisfy the condition \eqref{star_condition}.

For $\y = (y_1, \ldots, y_n)$ and $\x = (x_1, \ldots, x_m)$, we write

$$
|\y|_{\alf} = \max\limits_{1 \leq i \leq n} \alpha_i^{-1}(|y_i|) \,\,\,\,\,\,\,\,\,\,\,\, \text{and} \,\,\,\,\,\,\,\,\,\,\,\, |\x|_{\betf} = \max\limits_{1 \leq j \leq m}\beta_j^{-1}(|x_j|).
$$

One can define $| \cdot |_{\alf}$ and $| \cdot |_{\betf}$ analogously. By $|| \cdot ||_{\alf}$ and $|| \cdot ||_{\betf}$ we will denote the corresponding "distance to the nearest integer" functions, that is, 

$$
||\y||_{\alf} = \min\limits_{\p \in \mathbb{Z}^n} |\y - \p|_{\alf} \,\,\,\,\,\,\,\,\,\,\,\, \text{and} \,\,\,\,\,\,\,\,\,\,\,\, ||\x||_{\betf} = \min\limits_{\q \in \mathbb{Z}^n} |\x - \q|_{\betf}.
$$

We will study homogeneous and inhomogeneous approximation with respect to $| \cdot |_{\alf}$ and $| \cdot |_{\betf}$ instead of supremum norms. Let $f$ and $g$ be two functions $\mathbb{R}_+ \rightarrow \mathbb{R}_+$, decreasing to zero.

\begin{definition}

\begin{enumerate}
    \item We say that the matrix $\Theta^\top$ is $(\uf, \betf, \alf)$-approximable if the inequality
    \begin{equation}
        || \Theta^\top \y ||_{\betf} \leq f(|\y|_{\alf})
    \end{equation}

    has infinitely many solutions $\y \in \mathbb{Z}^n  \setminus \{ \bf 0 \}$. We denote the set of $(\uf, \betf, \alf)$-approximable $m \times n$ matrices by ${\bf W}_{m,n}[\uf; \betf, \alf]$.

    \item We say that the pair $( \Theta, \pmb{\eta})$ is $(\ug, \alf, \betf)$-Dirichlet if the system of inequalities
    \begin{equation}
    \begin{cases}
        || \Theta  \q - \pmb{\eta} ||_{\alf} \leq g(T), \\
        |\q|_{\betf} \leq T
    \end{cases}
    \end{equation}

    has a solution $\q \in \mathbb{Z}^m  \setminus \{ \bf 0 \}$ for any $T \in \mathbb{R}_+$ large enough. We denote the set of $(\ug, \alf, \betf)$-Dirichlet pairs by $\widehat{\bf D}_{n,m}[\ug; \alf, \betf]$. By $\widehat{\bf D}_{n,m}^{\Theta}[\ug; \alf, \betf]$ we denote the set
    $$
    \{ \pmb{\eta}: \,\,\, ( \Theta, \pmb{\eta}) \in \widehat{\bf D}_{n,m}[\ug; \alf, \betf] \},
    $$

    that is, the section of the set $\widehat{\bf D}_{n,m}[\ug; \alf, \betf]$ for a fixed $\Theta$.
\end{enumerate}
\end{definition}

The definitions above directly generalize the classical definitions of $\psi$-approximable matrices (see \cite{KM99}, Section 8.1) and $\psi$-Dirichlet pairs (\cite{kwad}, Section 1.2); the notation $\widehat{\bf D}_{n,m}^{\Theta}$ was used in \cite{MN25}. We note that some authors use the name {\it uniformly approximable} (see \cite{CGGMS20}, Definition 2.3) or {\it uniform} (\cite{KMWW}, Definition 1.1) instead of Dirichlet. 

We can rewrite these definitions explicitly. We use the notation $\Theta_i$ for $i$-th row of the matrix $\Theta$, and $(\Theta^{\top})_j$ for $j$-th row of the matrix $\Theta^{\top}$.

\begin{enumerate}
    \item The matrix $\Theta^{\top}$ is $(\uf, \betf, \alf)$-approximable if the system of inequalities

    $$
    \begin{cases}
        ||\left(\Theta^{\top}\right)_j \y|| \leq \beta_j (f(T)), \,\,\,\,\,\,\,\, j = 1, \ldots, m ; \\
        |y_i| \leq \alpha_i(T), \,\,\,\,\,\,\,\,\,\,\,\,\,\,\,\,\,\,\,\,\,\,\,\,\,\,\,\,\,\,\,\,\,\,\,\,\,\,\,\, i = 1, \ldots, n
    \end{cases}
    $$

    has a solution $\y = (y_1, \ldots, y_n) \in \mathbb{Z}^n$ for an unbounded set of $T \in \mathbb{R}_+$.

    \item The pair $( \Theta, \pmb{\eta})$ is $(\ug, \alf, \betf)$-Dirichlet if the system of inequalities

    $$
    \begin{cases}
        ||\Theta_i \q - \eta_i || \leq \alpha_i (g(T)), \,\,\,\,\,\,\,\, i = 1, \ldots, n ; \\
        |q_j| \leq \beta_j(T), \,\,\,\,\,\,\,\,\,\,\,\,\,\,\,\,\,\,\,\,\,\,\,\,\,\,\,\,\,\,\,\,\,\,\,\,\,\, j = 1, \ldots, m 
    \end{cases}
    $$

    has a solution $\q = (q_1, \ldots, q_m) \in \mathbb{Z}^m$ for any $T \in \mathbb{R}_+$ large enough.
\end{enumerate}

\vskip+0.3cm

The most well-studied special case of this setup is approximations with weights, which can be obtained by assuming $\alpha_i$ and $\beta_j$ are powers of the argument $T$. We will discuss this setup later in Section \ref{weightssection}. A non-trivial example of functions $\alpha_i$ and $\beta_j$ for which our main theorems hold is shown in Section \ref{3examplesection}.

\subsection{Quasimultiplicativity}

Before formulating our results, we will introduce technical conditions we need to put on the functions $\alpha_i$ and $\beta_j$. Call an increasing function $h: \,\,\, \mathbb{R}_+ \rightarrow \mathbb{R}_+$ {\it quasimultiplicative} if there exists an $M > 1$ such that for all $|k| \gg 1$ one has

    \begin{equation}\label{quasidefinition1}
    c_1 h(M^k) \leq h(M^{k+1}) \leq c_2 h(M^k),
    \end{equation}

    where $1 < c_1 \leq c_2$ are some absolute constants.

    We analogously define quasimultiplicativity for decreasing functions, assuming that $c_1 \leq c_2 < 1$.
    
\vskip+0.3cm

Let us note that we can choose the parameters $c_1, c_2$ and $M$ uniformly for any finite collection of quasimultiplicative functions.

We can reformulate the definition as follows: a monotonic function $h: \,\,\, \mathbb{R}_+ \rightarrow \mathbb{R}_+$ is quasimultiplicative if and only if there exist two constants $k_1$ and $k_2$, either both strictly positive or both strictly negative, such that for any $T \in \mathbb{R}_+$ and for any $R$ large enough 

\begin{equation}\label{quasireformulation}
    R^{k_1} h(T) \leq h(RT) \leq R^{k_2} h(T).
\end{equation}

Using \eqref{quasireformulation}, the following can be directly checked:

\begin{obs}\label{quasiproperties}
\begin{itemize}
    \item If $h$ and $r$ are two quasimultiplicative functions, either both increasing or both decreasing, then $h + r$ and $h \cdot r$ are also quasimultiplicative functions;
    \item If $h$ is an increasing (decreasing) quasimultiplicative function, then $\frac{1}{h}$ is a decreasing (increasing) quasimultiplicative function.
    \item If $h$ and $r$ are both increasing or both decreasing quasimultiplicative functions, then $h \circ r$ and $r \circ h$ are increasing quasimultiplicative functions. If one of the two function is increasing and one is decreasing, then $h \circ r$ and $r \circ h$ are decreasing quasimultiplicative functions.
    \item If $h$ is quasimultiplicative, its inverse $h^{-1}$ is quasimultiplicative.
    \item if $h$ is a quasimultiplicative function for which 
    $$
    h(T) = \frac{1}{h \left( \frac{1}{T} \right)} \,\,\,\,\,\,\,\,\, \text{for all} \,\,\, T \in \mathbb{R}_+,
    $$
    then it is enough to check condition \eqref{quasidefinition1} only for positive or only for negative values of $k$. In particular, this holds true for the functions $\alpha_i$ and $\beta_j$ satisfying \eqref{star_condition}.
\end{itemize}
\end{obs}

\begin{example}
    Let us define a function $h$ for $T \gg 1$ by $h(T) = T^{a_0} \left( \log T \right)^{a_1} \ldots \left( \log \ldots \log T \right)^{a_s}$ where $a_0 > 0$. It is easy to see that any such function can be extended to a quasimultiplicative one, and construct other examples using properties above.

    A nontrivial example is described in Section \ref{3examplesection}.

\end{example}

Quasimultiplicativity is a natural condition to assume for the weight functions; see for example Theorem 2.7 from \cite{KW23}, which is a generalization of the classical Khintchine-Groshev theorem (\cite{G38}). Here we give a version of this statement, formulated in terms of the weight functions $\alf$ and $\betf$ described above:

\begin{thm}\label{general_Khgr}
    Let $f(T), \,\,\, T \geq 1$ be a non-increasing real valued function. Let $\alf, \betf$ be defined as in Section \ref{general def section}, and suppose the functions $\beta_j$ are quasimultiplicative. The set ${\bf W}_{m,n}[\uf; \betf, \alf]$ has zero measure if the series

    \begin{equation}\label{khgr series}
    \sum\limits_{k=1}^{\infty} k^{-1} \prod\limits_{j=1}^m \beta_j(f(k)) \cdot \prod\limits_{i=1}^n \alpha_i(k)
    \end{equation}

    converges, and full measure if it diverges.
\end{thm}

\subsection{Main result for pairs}

The main goal of this paper is to generalize Theorem \ref{kwadresult} to the setup with weight functions and provide a short and simple proof for this generalized 0-1 law. More precisely, we will show that the following statement is equivalent to Theorem \ref{general_Khgr}:

\begin{theorem}\label{main for pairs}
    Let $g(T), \,\,\, T \geq 1$ be a non-increasing real valued function. Let $\alpha_i, \beta_j, \,\, i = 1, \ldots, n; \,\,j = 1, \ldots, m$, defined as in Section \ref{general def section}, be quasimultiplicative functions. The set $\widehat{\bf D}_{n,m}[\ug; \alf, \betf]$ has zero measure if the series

    \begin{equation}\label{kwseries}
    \sum\limits_{l=1}^{\infty} \frac{1}{l \cdot \prod\limits_{j=1}^m \beta_j(l) \cdot \prod\limits_{i=1}^n \alpha_i(g(l))}
    \end{equation}

    diverges, and full measure if it converges.
\end{theorem}

\vskip+0.3cm

Theorem \ref{kwadresult} is now a special case of Theorem \ref{main for pairs} with $\alpha_i(T) = T^{\frac{1}{n}}$ and $\beta_j(T) = T^{\frac{1}{m}}$.

\vskip+0.3cm

\subsection{Results for the fixed matrix (twisted) case}

To prove Theorem \ref{main for pairs}, we will need two statements about the sets ${\widehat{\bf D}^{\Theta} [g; \alf, \betf]}$, which we believe are of a separate interest. 

The setup of Diophantine approximations of pairs $(\Theta, \pmb{\eta})$ with a fixed matrix $\Theta$ is also known as {\it twisted} Diophantine approximation (see \cite{Ha12}), and the metric theory of twisted asymptotic Diophantine approximations was studied before; see for example \cite{HW24} for weighted setup and a recent work \cite{BSV25} studying vectors on certain manifolds. Less has been done for the case of uniform twisted approximations; see some recent progress in \cite{MN25}.

It is clear that the structure of the sets ${\widehat{\bf D}^{\Theta} [g; \alf, \betf]}$ depends on some Diophantine properties of the matrix $\Theta$. A more specific relation is given by the transference principle --- a classical technique, introduced almost a century ago by Khintchine (\cite{khint}) and adapted for inhomogeneous problems by Jarnik (\cite{j39, j41, j46}). In our case, the transference principle provides a connection between uniform inhomogeneous approximations of pairs $(\Theta, \pmb{\eta})$ and asymptotic homogeneous approximations of the transposed matrix $\Theta^{\top}$.

Hereafter we will often assume that the functions $f$ and $g$ are strictly decreasing, continuous and related by 

\begin{equation}\label{fgrelation}
    g(T) = \frac{1}{f^{-1} \left( \frac{1}{T} \right)}.
\end{equation}

Our first statement about "slices" can be thought of as an analog of Dirichlet's theorem for the fixed matrix case:

\begin{theorem}\label{Jarnik_general}
    Let $\alpha_i, \beta_j, \,\, i = 1, \ldots, n; \,\,j = 1, \ldots, m$, defined as in Section \ref{general def section}, be some increasing continuous functions. Suppose $f$ is such a continuous, strictly decreasing function that $\Theta^T \notin {\bf W}_{m,n}[\uf; \betf, \alf]$, and let $\uf$ and $\ug$ be related by \eqref{fgrelation}. Let $C = C_d: = d \cdot d!$, where $d = n+m$. Then, 

    $$
    \widehat{\bf D}_{n,m}^{\Theta}\left[\ug; C \alf, C\betf \right]  = \mathbb{R}^n.
    $$             

\end{theorem}

The next statement, closely following Theorem 3 from \cite{MN25}, is a twisted analog of the classical theorem by Davenport and Schmidt (\cite{DS1, DS2}): it says that the set of "$\varepsilon$-Dirichlet improvable" pairs $(\Theta, \pmb{\eta})$ for a fixed $\Theta$ has Lebesgue measure zero. 

We use the notation $\circ$ to denote the composition of functions; in particular, $\ug \circ \frac{1}{\varepsilon} (T) = g(\frac{1}{\varepsilon} T)$.

\begin{theorem}\label{weighted_impr}
    Suppose $f$ is such a continuous, strictly decreasing function that $\Theta^T \in {\bf W}_{m,n}[\uf; \betf, \alf]$, and let $\uf$ and $\ug$ be related by \eqref{fgrelation}. Suppose all the functions $\alpha_i, \beta_j$, defined as in Section \ref{general def section}, are quasimultiplicative. Let $\tilde{g} = \tilde{g}_{\varepsilon}: = \varepsilon \cdot \ug \circ \frac{1}{\varepsilon}$. If $\varepsilon$ is small enough, then the set $\widehat{\bf D}_{n,m}^{\Theta}[\tilde{g}; \alf, \betf]$ has Lebesgue measure zero.
\end{theorem}

\begin{remark}
    The statement of Theorem \ref{weighted_impr}, as well as Proposition \ref{DI with weights}, holds also for $\pmb{\eta}$ restricted on certain {\it non-exceptional} affine subspaces of $\mathbb{R}^n$; almost all the affine subspaces are non-exceptional. The discussion of exceptional subspaces can be found in \cite{MN25}; the proof is analogous. 
\end{remark}

\vskip+0.3cm

\subsection{A special case: approximations with weights}\label{weightssection}

A straightforward application of our results is the classical setup of Diophantine approximations with weights.

Let $\rho_1, \ldots, \rho_n$ and $\sigma_1, \ldots, \sigma_m$ be an $n$-tuple and $m$-tuple of positive real numbers satisfying 

\begin{equation}\label{weightsdef}
\sum\limits_{i=1}^n \rho_i = \sum\limits_{j=1}^m \sigma_j = 1.
\end{equation}

These tuples define the quasinorms $| \cdot |_{\rrho}$ and $| \cdot |_{\ssigma}$ on $\mathbb{R}^n$ and $\mathbb{R}^m$ by

$$
| {\bf y} |_{\rrho} = \max\limits_{i=1, \ldots, n} |y_i|^{\frac{1}{\rho_i}} \,\,\,\,\,\,\,\,\,\,\,\,\,\,\,\, \text{and} \,\,\,\,\,\,\,\,\,\,\,\,\,\,\,\, | {\bf x} |_{\ssigma} = \max\limits_{j=1, \ldots, m} |x_j|^{\frac{1}{\sigma_j}}.
$$

We will also denote by $|| \cdot ||_{\rrho}$ and $|| \cdot ||_{\ssigma}$ the distance to the nearest integer functions, corresponding to $| \cdot |_{\rrho}$ and $| \cdot |_{\ssigma}$. The subject of inhomogeneous Diophantine approximations with weights is the solvability of systems 

\begin{equation}\label{weighted_system}
    \begin{cases}
        || \Theta  \q - \pmb{\eta} ||_{\rrho} \leq g(T), \\
        |\q|_{\ssigma} \leq T
    \end{cases}
    \end{equation}

of inequalities in integral vectors $\q \in \mathbb{Z}^m$. We say that a pair $(\Theta, \pmb{\eta})$ is { \it $(g; \rrho, \ssigma)$-Dirichlet } if the system \ref{weighted_system} has a solution $\q \in \mathbb{Z}^m$ for any $T$ large enough.
Abusing the notation, we will denote the set of $(g; \rrho, \ssigma)$-Dirichlet pairs by ${\widehat{\bf D} [g; \rrho, \ssigma]}$ when it is clear that $\rrho$ and $\ssigma$ are weights, not functions; we will use the notation ${\widehat{\bf D}^{\Theta} [g; \rrho, \ssigma]}$ for sections with a fixed matrix $\Theta$.

\vskip+0.3cm

    We say that the matrix $\Theta^T$ is $(f, \ssigma, \rrho)$-approximable if the inequality
    \begin{equation}
        || \Theta^T y ||_{\ssigma} \leq f(|y|_{\rrho})
    \end{equation}

    has infinitely many solutions $y \in \mathbb{Z}^n$. Analogously, we denote the set of $(f, \ssigma, \rrho)$-approximable $m \times n$ matriecs by ${\bf W}_{m,n}\left[f, \ssigma, \rrho \right]$, and the sections with a fixed $\Theta$ by ${\bf W}_{m,n}^{\Theta}\left[f, \ssigma, \rrho \right]$.

\vskip+0.3cm

By choosing $\alpha_i(T): = T^{\rho_i}$ and $\beta_j(T): = T^{\sigma_j}$ in Theorem \ref{main for pairs}, we obtain the weighted version of Theorem \ref{kwadresult}:

\begin{proposition}\label{kwad_weighted}
    Let $g(T), \,\,\, T \geq 1$ be a non-increasing real valued function. The set of $(g, \rrho, \ssigma)$-Dirichlet pairs has zero Lebesgue measure if the series

    $$
    \sum\limits_{l=1}^{\infty} \frac{1}{l^2 g(l)}
    $$

    diverges, and full measure if it converges.
\end{proposition}

We note that for the functions $g$, satisfying additional technical restrictions, the statement of Proposition \ref{kwad_weighted} can also be obtained as a direct corollary of Theorem 1.6 from \cite{He25}.

\vskip+0.3cm

The version of Theorem \ref{Jarnik_general} for weights is not new; it can be found in \cite{CGGMS20} as Lemma 2.4. 

\begin{prop}\label{lem24}
     Suppose $f$ is such a continuous, strictly decreasing function that $\Theta^{\top}$ is not $(f, \ssigma, \rrho)$-approximable, and let $\uf$ and $\ug$ be related by \eqref{fgrelation}. For $C$ large enough, any pair $(\Theta, \pmb{\eta})$ is $(C \cdot \ug \circ \frac{1}{C}, \rrho, \ssigma)$-Dirichlet, that is, 

    $$
    \widehat{\bf D}_{n,m}^{\Theta}[C \cdot \ug \circ \frac{1}{C}; \rrho, \ssigma]  = \mathbb{R}^n.
    $$   
\end{prop}

Finally, Theorem \ref{weighted_impr} yields an inhomogeneous version of the Davenport-Schmidt theorem for the fixed matrix case with an explicit bound for the constant $\varepsilon$.

\begin{proposition}\label{DI with weights}
    Suppose $f$ is such a continuous, strictly decreasing function that $\Theta^{\top}$ is $(f, \ssigma, \rrho)$-approximable, and let $\uf$ and $\ug$ be related by \eqref{fgrelation}. If $\varepsilon$ is small enough, then the set $\widehat{\bf D}_{n,m}^{\Theta}[\varepsilon \cdot g \circ \frac{1}{\varepsilon} ; \rrho, \ssigma]$ has Lebesgue measure zero. More specifically, let $r_- = \min\limits_{i,j} \{ \rho_i, \sigma_j \}$; then it is enough to require $\varepsilon < {(2d)^{-\frac{1}{r_-}}}$.
\end{proposition}

\vskip+0.3cm

Proposition \ref{kwad_weighted}, Proposition \ref{lem24} and Proposition \ref{DI with weights} provide a complete Lebesgue measure theory for sets of $(\ug, \rrho, \ssigma)$-Dirichlet pairs $(\Theta, \pmb{\eta})$ and the slices with a fixed matrix $\Theta$ for any weights $\rrho$ and $\ssigma$. The remaining case of the sets of $(\ug, \rrho, \ssigma)$-Dirichlet pairs $(\Theta, \pmb{\eta})$ with a fixed vector $\pmb{\eta}$ has also been studied: a zero-one law, similar to Proposition \ref{kwad_weighted}, is obtained for a non-weighted setting in \cite{kimkim} and then in \cite{He25} (Theorem 1.6) for "well-behaved" functions $g$ for Lebesgue measure and, generally, any Hausdorff measure in the weighted case.

\vskip+0.3cm

\subsection{A nontrivial example: approximations with changing weights}\label{3examplesection}

In this section we show a simple but nontrivial example of quasimultiplicative function tuples, showing that the results of our paper apply not only in the setup of approximations with weights or slight generalizations (such as adding logarithmic factors or taking polynomials into consideration), but in a variety of significantly different situations.

Let $\theta_1, \theta_2$ and $\eta$ be real numbers, and consider the system of inequalities 

\begin{equation}\label{3dim example}
\begin{cases}
    || \theta_1 q_1 + \theta_2 - \eta || \leq g(T), \\
    |q_1| \leq \beta_1(T),\\
    |q_2| \leq \beta_2(T),
\end{cases}
\end{equation}

corresponding to approximations of one linear form with two variables.

We will construct a pair of quasimultiplicative functions $\beta_1, \beta_2$ with the following property: there exists an unbounded set of values of $T$, for which the system \eqref{3dim example} corresponds to the setup with weights $(\sigma_1, \sigma_2) = (\frac{1}{3}, \frac{2}{3})$ (that is, $\beta_1(T) =f_{-\frac{1}{3}}(T) =  T^{\frac{1}{3}}$ and $\beta_2(T) =f_{-\frac{2}{3}}(T) =  T^{\frac{2}{3}}$), and an unbounded set of values of $T$, for which it reflects the flipped weights $(\sigma_1, \sigma_2) = (\frac{2}{3}, \frac{1}{3})$ (that is, $\beta_1(T) = T^{\frac{2}{3}}$ and $\beta_2(T) = T^{\frac{1}{3}}$).

It will be easier to describe these functions if we switch to the logarithmic scale. Let us define the functions $\gamma_j(t): = \ln \beta_j(e^t)$, that is, $\beta_j(T) = e^{\gamma_j \left( \ln T \right)}$, and describe the functions $\beta_j$ via $\gamma_j$. 

Note that $T^{\frac{1}{3}} = e^{\varphi_1 \ln T}$ and $T^{\frac{2}{3}} = e^{\varphi_2 \ln T}$, where $\varphi_1(t) = \ln f_{-\frac{1}{3}}(e^t) = \frac{1}{3} t$ and $\varphi_2(t) = \frac{2}{3} t$. We will construct $\gamma_1(t)$ as a piecewise linear functions with slopes $\frac{1}{4}$ or $\frac{3}{4}$, oscillating between $\varphi_1(t)$ and $\varphi_2(t)$, and $\gamma_2(t)$ as the function with the same property defined by $\gamma_1(t) + \gamma_2(t) = t$.

For instance, let us define for $t \geq 1$ 

$$
\gamma_1(t): = 
\begin{cases}
    \frac{3}{4}t -\frac{5}{12} \cdot  5^{2k}, \,\,\,\,\,\,\,\,\,\,\,\, 5^{2k} \leq t < 5^{2k+1};\\
    \frac{1}{4}t + \frac{5}{12} \cdot 5^{2k+1}, \,\,\,\,\,\,\,\,\,\,\,\, 5^{2k+1} \leq t < 5^{2k+2},
\end{cases}
$$

with $\gamma_1(t) = \frac{1}{3}t$ for $0 \leq t \leq 1$, and

$$
\gamma_2(t): = t - \gamma_1(t) =  
\begin{cases}
    \frac{1}{4}t + \frac{5}{12} \cdot  5^{2k}, \,\,\,\,\,\,\,\,\,\,\,\, 5^{2k} \leq t < 5^{2k+1};\\
    \frac{3}{4}t - \frac{5}{12} \cdot 5^{2k+1}, \,\,\,\,\,\,\,\,\,\,\,\, 5^{2k+1} \leq t < 5^{2k+2}.
\end{cases}
$$

The graphs of ${\gamma_1}, {\gamma_2}, {\varphi_1}$ and ${\varphi_2}$ are shown at Figure~\ref{Fig:Data1} below.


\begin{figure}[!ht]
   \begin{minipage}{\textwidth}
     \centering
     \includegraphics[width=.6\linewidth]{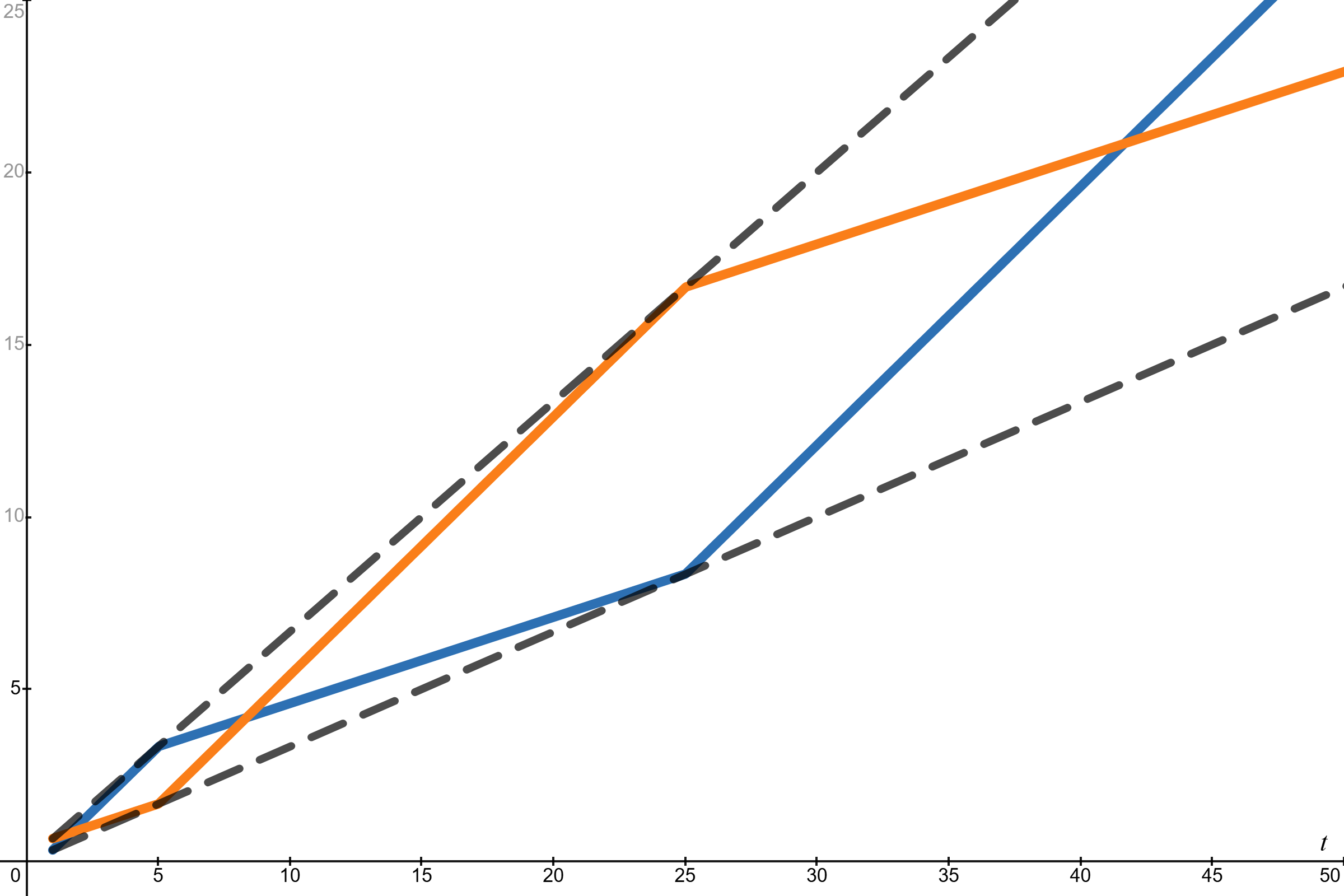}
     \caption{${\gamma_1} \, \text{(blue)}, {\gamma_2} \, \text{(orange)}, {\varphi_1}$ and ${\varphi_2}$ (black)}\label{Fig:Data1}
   \end{minipage}
\end{figure}

This defines the values $\beta_1(T)$ and $\beta_2(T)$ for $T \geq 1$, such that $\beta_1(1) = \beta_2(1) = 1$; we extend these functions to $\mathbb{R}_+$ by \eqref{star_condition}.
It is easy to check quasiultiplicativity of $\beta_1$ and $\beta_2$: indeed

$$
\frac{\beta_j(e^{k+1})}{\beta_j(e^{k})} = e^{\gamma_j(k+1) - \gamma_j(k)} \in [ e^\frac{1}{4}; e^\frac{3}{4}],
$$

so one can take $c_1 = e^\frac{1}{4}, \, c_2 = e^\frac{3}{4}$ and $M = e$.

\subsection{Structure of the paper}

{

In Section \ref{geom section} we introduce a useful geometric setup and reformulate the main definitions in terms of trajectories on the space of lattices. Section \ref{applsection} studies a generalized set of badly approximable matrices in the setup of general weight functions, and provides a criterion of bad approximability in terms of inhomogeneous approximations; in addition, we provide some more specific statements regarding badly and very well approximable matrices for the setup of approximations with weights in Section \ref{weights_bad_section}. We prove our main statements in Section \ref{section proofs}.

}

\section{On geometry and dynamics}\label{geom section}
\subsection{Relation to geometry of numbers}\label{generalgeom}

In this section, we translate our setup to geometrical language and recall both classical and more recent results concerning geometry of numbers which will be used in our proofs.

We denote by

$$
X_d: = GL_d \left( \mathbb{R} \right)/GL_d \left( \mathbb{Z} \right), \,\,\,\,\,\,\,\,\,\,\,\,\,\,\, \widehat{X}_d: = \left( GL_d \left( \mathbb{R} \right) \rtimes \mathbb{R}^d \right)/\left( GL_d \left( \mathbb{Z} \right) \rtimes \mathbb{Z}^d \right)
$$

the spaces of lattices and grids (affine lattices) in $\mathbb{R}^d$, and by

$$
X^1_d: = SL_d \left( \mathbb{R} \right)/SL_d \left( \mathbb{Z} \right), \,\,\,\,\,\,\,\,\,\,\,\,\,\,\, \widehat{X}^1_d: = \left( SL_d \left( \mathbb{R} \right) \rtimes \mathbb{R}^d \right)/\left( SL_d \left( \mathbb{Z} \right) \rtimes \mathbb{Z}^d \right)
$$

the spaces of unimodular lattices and grids. We note that $GL_d \left( \mathbb{R} \right)$ acts on $X_d$ by multiplication and, analogously, $ GL_d \left( \mathbb{R} \right) \rtimes \mathbb{R}^d $ acts on $\widehat{X}_d$. One can obtain an action of $GL_d \left( \mathbb{R} \right)$ on $\widehat{X}_d$, induced from the action of  $ GL_d \left( \mathbb{R} \right) \rtimes \mathbb{R}^d $ by identifying an element $A \in   GL_d \left( \mathbb{R} \right)$ with $\langle A, {\bf 0} \rangle \in  GL_d \left( \mathbb{R} \right) \rtimes \mathbb{R}^d $.

Let

$$
u_{\Theta} = 
\begin{pmatrix}
    I_n & \Theta \\
    0 & I_m
\end{pmatrix}
\,\,\,\,\,\,\,\,\,\,\,\, \text{and} \,\,\,\,\,\,\,\,\,\,\,\,
u_{\Theta, \pmb{\eta}} = \left\langle \begin{pmatrix}
    I_n & \Theta \\
    0 & I_m
\end{pmatrix}, \begin{pmatrix}
    - \pmb{\eta} \\
    0 
\end{pmatrix} \right\rangle,
$$

where $I_n$ is the $n$-dimensional identoty matrix. We will utilize the notation $u_{-\Theta}^{\top}$ for the transposed operator $\begin{pmatrix}
    I_n & 0 \\
    -\Theta^{\top} & I_m
\end{pmatrix}$.

We define 

$$
\Lambda_{\Theta}: = u_{\Theta} \mathbb{Z}^d = \left\{  \begin{pmatrix}
    \Theta \q  + \p \\
    \q
\end{pmatrix}, \,\,\,\p \in \mathbb{Z}^n, \,\, \q \in \mathbb{Z}^m \right\} \in X_d
$$

and 

$$
\Lambda_{\Theta, \pmb{\eta}}: = u_{\Theta, \pmb{\eta}} \mathbb{Z}^d = \left\{  \begin{pmatrix}
    \Theta \q - \pmb{\eta} + \p \\
    \q
\end{pmatrix}, \,\,\,\p \in \mathbb{Z}^n, \,\, \q \in \mathbb{Z}^m \right\} \in \widehat{X}_d,
$$

as well as 

$$
\Lambda_{-\Theta}^{\top}: = u_{-\Theta}^{\top} \mathbb{Z}^d = \left\{  \begin{pmatrix}
    \p \\
    -\Theta^{\top} \p + \q
\end{pmatrix}, \,\,\,\p \in \mathbb{Z}^n, \,\, \q \in \mathbb{Z}^m \right\} \in X_d.
$$

\vskip+0.3cm

Let $\alpha_i, \beta_j$ be as in Section \ref{general def section}. We define the family of diagonal operators

$$
a_{\alf, \betf} \left( U, T \right) = diag \left( \frac{1}{\alpha_1(U)}, \ldots, \frac{1}{\alpha_n(U)}, \frac{1}{\beta_1(T)}, \ldots, \frac{1}{\beta_m(T)} \right) \in GL_d \left( \mathbb{R} \right),
$$

with

$$
a_{\alf, \betf}^{-1} \left( U, T \right) = diag \left( {\alpha_1(U)}, \ldots, {\alpha_n(U)}, {\beta_1(T)}, \ldots, {\beta_m(T)} \right) = a_{\alf, \betf} \left( \frac{1}{U}, \frac{1}{T} \right).
$$

being the family of inverse operators. From now on, we will refer to the cube $[-1, 1]^d$ as $\mathcal{B}_d$. One can directly check the following useful facts:

\begin{obs}\label{geom_sets}{(Dani's correspondence; see \cite{Dani} and Theorem 8.5 from \cite{KM99} for classical analogues.)}

\begin{itemize}
    \item $\Theta \in {\bf W}_{n,m}[\ug; \alf, \betf]$ if any only if

    $$
    a_{\alf, \betf} \left( g(T), T \right) \Lambda_{\Theta} \cap \mathcal{B}_d \neq \{ {\bf 0} \}
    $$

    for an unbounded set of $T \in \mathbb{R}_+$;
    
    \item If $\pmb{\eta} \neq {\bf 0}$, then $(\Theta, \pmb{\eta}) \in \widehat{\bf D}_{n,m}[\ug; \alf, \betf]$ if any only if

    $$
    a_{\alf, \betf} \left( g(T), T \right) \Lambda_{\Theta, \pmb{\eta}} \cap \mathcal{B}_d \neq \emptyset
    $$

    for any $T \in \mathbb{R}_+$ large enough. In the case $\pmb{\eta} = {\bf 0}$, we should require the aforementioned set to contain nonzero points.

    \item $\Theta^{\top} \in {\bf W}_{m,n}[\uf; \betf, \alf]$ if any only if

    $$
    a_{\alf, \betf} \left( U, f(U) \right) \Lambda_{-\Theta}^{\top} \cap \mathcal{B}_d \neq \{ {\bf 0} \}
    $$

    for an unbounded set of $U \in \mathbb{R}_+$.
\end{itemize}

\end{obs}

We now introduce the notions of dual lattices (also called {\it polar}; see \cite{cas} for details and discussion) and pseudo-compound parallelepipeds (introduced by Schmidt as a simplification of Mahler's compound bodies; see \cite{schmidtbook} for details).

\begin{definition}
    \begin{itemize}







        \item Let $\Lambda \in X_d$ be a lattice in $\mathbb{R}^d$. The set

        $$
        \Lambda^* = \{ \y \in \mathbb{R}^d: \,\,\, \x \cdot \y \in \mathbb{Z} \,\,\,\, \text{for any} \,\,\, \x \in \Lambda \}
        $$

        is a  lattice in $\mathbb{R}^d$, called {\it the dual lattice} for $\Lambda$.

        \item Let 

        \begin{equation}\label{standard}
        \mathcal{P} = \{ \x = (x_1, \ldots, x_d) \in \mathbb{R}^d: \,\,\, |x_i| \leq \lambda_i, \,\,\, i = 1, \ldots, d \},
        \end{equation}

        where $\lambda_1, \ldots, \lambda_d$ are positive real numbers, be a parallelepiped, and $L = \prod\limits_{j=1}^d \lambda_j$. Then, the parallelepiped 

        $$
        \mathcal{P}^* = \{ \y = (y_1, \ldots, y_d) \in \mathbb{R}^d: \,\,\, |y_i| \leq \frac{L}{\lambda_i}, \,\,\, i = 1, \ldots, d \},
        $$

        is called {\it pseudo-compound} for $\mathcal{P}$.
    \end{itemize}
\end{definition}


\vskip+0.3cm

The following useful statement follows Mahler's works. Part \ref{kostya} is a strengthening of Mahler's original theorem which was shown in \cite{GK15} (Theorem 1; we use it in the formulation of Theorem 8 from \cite{G20}). Part \ref{mahler} was implicitly stated by Mahler (see the review \cite{Evertse} for details) and, with an explicit constant, for general dual bodies as Lemma 2.1 in \cite{CGGMS20}. Our statement slightly improves this constant for the case of pseudo-compound parallelepipeds. 

\begin{lem}\label{lem21}
\begin{enumerate}[label=(\arabic*)]

    \item\label{kostya} Let $c = c_d: = d^{\frac{1}{2(d-1)}}$. Let $\Lambda$ be a lattice in $\mathbb{R}^d$, and $\Lambda^*$ its dual. Let $\mathcal{P}$ be a standard parallelepiped of the form \eqref{standard}, and $\mathcal{P}^*$ be its pseudo-compound. If
    $$
    \mathcal{P}^* \cap \Lambda^* \neq \{ {\bf 0} \},
    $$

    then

    $$
    c \mathcal{P} \cap \Lambda \neq \{ {\bf 0} \}.
    $$
    
    \item\label{mahler} Let $C = C_{d} = d \cdot d!$.  Let $\Lambda \in X_d$, $\mathcal{P}$ be a standard parallelepiped of the form \eqref{standard} in $\mathbb{R}^d$, and $\Lambda^*$ and $\mathcal{P}^*$ be the dual lattice and pseudo-compound parallelepiped for $\Lambda$ and $\mathcal{P}$, respectfully. Suppose that $\mathcal{P}^* \cap \Lambda^* = \{ { \bf 0 }\}$. Then for any $\gamma \in \mathbb{R}^d$ one has

    $$
    \left(\frac{C}{L} \mathcal{P} + \gamma \right) \cap \Lambda \neq \emptyset,
    $$

    moreover, 

    $$
    \frac{C}{L} \mathcal{P}\cap \Lambda \neq\{ { \bf  0 } \}.
    $$

\end{enumerate}
\end{lem}

We give a proof for part \ref{mahler} below.

{

\begin{proof}
    Let $\mu_k(\mathcal{P}, \Lambda)$ denote the $k$-th consecutive minimum of $\mathcal{P}$ with respect to the lattice $\Lambda$. By assumption, $ \mu_1(\mathcal{P}^*, \Lambda^*) > 1$. By classical result due to Mahler (we are using it in the form of Theorem D from \cite{GK15}), 

    $$
    \mu_1(\mathcal{P}^*, \Lambda^*)  \mu_d(\mathcal{P}, \Lambda) \leq \frac{2^d d!}{vol \mathcal{P}} = \frac{d!}{L},
    $$

    which implies that the parallelepiped $\frac{d!}{L} \mathcal{P}$ contains $d$ linearly independent points of lattice $\Lambda$. It is easy to see that $d \cdot \frac{d!}{L} \mathcal{P}$ then contains a closure of some fundamental domain of $\Lambda$, which implies the desired statement.
\end{proof}

}

\vskip+0.3cm

 Now one can make the following useful observations:

\begin{obs}\label{obs_cube}
    The cube $\mathcal{B}_d$ is self-pseudo-compound: $\mathcal{B}_d^* = \mathcal{B}_d$. More generally, $\left( k \mathcal{B}_d\right)^* = k^{d-1} \mathcal{B}_d$.
\end{obs}

\begin{obs}\label{obs_dual}
    The lattice 
    $$
    \Lambda_{\Theta}^*: = \Lambda_{-\Theta}^{\top} = u_{-\Theta}^{\top} \mathbb{Z}^d
    $$

    is a dual lattice for $\Lambda_{\Theta}$, and the lattice

    $$
    a_{\alf, \betf} \left( U, T \right)\Lambda_{\Theta}^* = a_{\alf, \betf} \left( U, T \right) u_{-\Theta}^{\top} \mathbb{Z}^d
    $$

    is a dual for $ a_{\alf, \betf} \left( \frac{1}{U}, \frac{1}{T} \right) \Lambda_{\Theta}$.
\end{obs}

\vskip+0.3cm

\subsection{A special case: approximations with weights}

For convenience of further use, we will slightly modify the statement of Observation \ref{geom_sets} for the case of approximations with weights (Section \ref{weightssection}); we only formulate it for the functions of the form $b  f_1$, where $b>0$ is a constant.

Let $\rho_i, \sigma_j$ be as in \ref{weightsdef}, and suppose $\alpha_i(T) = T^{\rho_i}; \beta_j(T) = T^{\sigma_j}$. We set

$$
a_{\rho, \sigma}(T): = a_{\alf, \betf}\left(T^{-1}, T \right) = diag(T^{\rho_1}, \ldots, T^{\rho_n}, T^{- \sigma_1}, \ldots, T^{-\sigma_m}) \in SL_d(\mathbb{R}).
$$

We also define the parallelepipeds $P_{\rrho, \sigma}(b)$ by

$$
    P_{\rrho, \sigma}(b): = \{ {\bf v} = (v_1, \ldots, v_d) \in \mathbb{R}^d: \,\,\,\, |v_i| \leq b^{\frac{\rho_i}{2}}, \, |v_{n+j}| \leq b^{\frac{\sigma_j}{2}}, \,\,\, i = 1, \ldots,. n; \,\, j = 1, \ldots, m \}.
    $$

     A direct calculation shows that 

    $$
      P_{\rrho, \sigma}(b)^*  = \{ {\bf v} = (v_1, \ldots, v_d) \in \mathbb{R}^d: \,\,\,\, |v_i| \leq b^{\frac{2-\rho_i}{2}}, \, |v_{n+j}| \leq b^{\frac{2-\sigma_j}{2}}, \,\,\, i = 1, \ldots, n; \,\, j = 1, \ldots, m \}
    $$

    and $\left( \frac{1}{c} P_{\rrho, \sigma}(z)\right)^* = \frac{1}{c^{d-1}} P_{\rrho, \sigma}(z)^*$.

The criterion below follows directly from Observation \ref{geom_sets}.

\begin{obs}\label{logarithmic criterion}
    
    \begin{enumerate}
    \item The matrix $\Theta \in {\bf M}_{n,m}$ is $(bf_1; \rho, \sigma)$-approximable if and only if

    $$
    a_{\rho, \sigma}(T) \Lambda_{\Theta} \cap P_{\rrho, \sigma}(b) \neq \{ {\bf 0} \}
    $$

    for an unbounded from above set of $T$.
    
        \item\label{inhom_unif} The pair $(\Theta, \pmb{\eta})$ is $(bf_1; \rho, \sigma)$-Dirichlet if and only if

    $$
    a_{\rho, \sigma}(T) \Lambda_{\Theta, \pmb{\eta}} \cap P_{\rrho, \sigma}(b) \neq \{ {\bf 0} \}
    $$

    for all $T$ large enough.

    \item The matrix $\Theta^{\top} \in {\bf M}_{m,n}$ is $(bf_1; \ssigma, \rrho)$-approximable if and only if

    $$
    a_{\rho, \sigma}\left(\frac{1}{T}\right) \Lambda_{-\Theta}^{\top} \cap P_{\rrho, \sigma}(b) \neq \{ {\bf 0} \}
    $$

    for an unbounded from above set of $T$.
    \end{enumerate}
\end{obs}

\section{Applications}\label{applsection}

\subsection{Homogeneous approximations: Dirichlet's theorem}\label{Dirichlet}

To add some context, we will start with establishing a simple analog of the homogeneous Dirichlet's theorem for our setup with the weighted functions.

Suppose $\alpha_i, \beta_j$ are as in Section \ref{general def section}. Let $\alpha(T) : = \prod\limits_{i=1}^n \alpha_i(T)$ and $\beta(T): = \prod\limits_{j=1}^m \beta_j(T)$, and define the function $g_{\alf, \betf}(T): = \alpha^{-1} \left( \frac{1}{\beta(T)} \right)$. 

\begin{theorem}\label{Dirichlet homogeneous}
    Any matrix $\Theta \in { \bf M}_{n,m}$ is $(g_{\alf, \betf}, \alf, \betf)$-Dirichlet, and thus is $(g_{\alf, \betf}, \alf, \betf)$-approximable. 
\end{theorem}

\begin{proof}
    Note that $m \left( \mathcal{B}_d \right) = 2^d$, and the lattice $a_{\alf, \betf} \left(g_{\alf, \betf}(T), T \right) \Lambda_{\Theta}$ is unimodular; now, the desired statement follows from Observation \ref{geom_sets} and Minkowski's convex body theorem.
\end{proof}

Let us notice that $g_{\alf, \betf}$ and $g_{\betf, \alf}$ satisfy a relation similar to \eqref{fgrelation}: namely, 

$$
g_{\betf, \alf}(T) = \frac{1}{g_{\alf, \betf} \left( \frac{1}{T} \right)}.
$$

\vskip+0.3cm

In the case of approximations with weights ($\alpha_i(T) = T^{\rho_i}, \, \beta_j(T) = T^{\sigma_j}$), due to the relation \eqref{weightsdef}, one has $g_{\betf, \alf}(T) = f_1(T) = \frac{1}{T}$ and so Theorem \ref{Dirichlet homogeneous} is the classical Dirichlet's theorem.

\subsection{Bad approximability: general setup}\label{bad_general}

The statement of Theorem \ref{Dirichlet homogeneous} motivates the following generalization of the classical definition:

\begin{definition}\label{generalbaddef}
     Matrix $\Theta \in {\bf M}_{n,m}$ is called { \it $(\alf, \betf)$-badly approximable} if it is not $(b \cdot g_{\alf, \betf}, \alf, \betf)$-approximable for some constant $b > 0$.  
\end{definition}

We denote the set of $(\alf, \betf)$-badly approximable $n \times m$ matrices by ${\bf BA}_{n,m}[\alf, \betf]$. We note that if $\alpha_i, \beta_j$ are quasimultiplicative, then by Theorem \ref{general_Khgr} the set ${\bf BA}_{n,m}[ \alf, \betf]$ is a Lebesgue nullset. One more analogy with the classical set of badly approximable matrices can be seen through a dynamical interpretation. 

\vskip+0.3cm

For a map $\Lambda(T): \,\, \mathbb{R}_+ \rightarrow X^1_d$, we call the trajectory ${\Lambda(T)}_{T \in \mathbb{R}_+}$ { \it bounded} if it is contained in a bounded subset of $X^1_d$. We define the sets

$$
X_d^1(\lambda): = \left\{ \Lambda \in X_d^1: \,\,\,\,  \Lambda \cap \lambda \mathcal{B}_d = \{ {\bf 0} \} \right\}.
$$

By Mahler's compactness criterion, the sets $X_d^1(\lambda)$ are compact in $X_d^1$, and any bounded subset of $X_d^1$ is contained in $X_d^1(\lambda)$ for some $\lambda > 0$; therefore, a trajectory is bounded if and only if there exists such a $\lambda>0$ that $\Lambda(T) \in X_d^1(\lambda)$ for all $T$ large enough.

Using Observation \ref{geom_sets} and quasimultiplicativity of functions $\alpha_i$ and $\beta_j$, one can show that $(\alf, \betf)$-badly approximable matrices are in a correspondence with bounded trajectories in $X^1_d$, thus generalizing a similar result for approximations with weights due to Kleinbock (\cite{K98}, Theorem 2.5):

\begin{proposition}\label{bad=bounded}
    Suppose all the functions $\alpha_i$ and $\beta_j$, defined as in Section \ref{general def section}, are quasimultiplicative. Matrix $\Theta \in {\bf M}_{n,m}$ is $(\alf, \betf)$-badly approximable if and only if the trajectory

    $$
    \left\{  a_{\alf, \betf} \left( g_{\alf, \betf}(T), T \right) \Lambda_{\Theta} \right\}_{T \in \mathbb{R}_+}
    $$

    is bounded in $X^1_d$.
\end{proposition}

\begin{proof}
    Let $g = g_{\alf, \betf}$. First, let us note that by Observation \ref{quasiproperties} the function $g_{\alf, \betf}$ is quasimultiplicative. Fix such $1< k_1 \leq k_2$ that the functions $\alpha_i$ and $\beta_j$ satisfy \eqref{quasireformulation} with these constants, and $g$ together with $\alpha_i \circ g$ satisfy \eqref{quasireformulation} with constants $-k_2$ and $-k_1$.

    \begin{itemize}
        \item Suppose $\Theta$ is $(\alf, \betf)$-badly approximable: by Observation \ref{geom_sets} there exists such a $b$ that 

        $$
    a_{\alf, \betf} \left( b g(T), T \right) \Lambda_{\Theta} \cap \mathcal{B}_d = \{ {\bf 0} \}.
    $$

    By quasimultiplicativity of $\alpha_i$, the diagonal operator

    $$
    a_{\alf, \betf} \left( g(T), T \right) a^{-1}_{\alf, \betf} \left( b g(T), T \right)
    $$

    contracts not more than by $b^{k_2}$ in each direction, which implies

    $$
    a_{\alf, \betf} \left(  g(T), T \right) \Lambda_{\Theta} \cap b^{k_2} \mathcal{B}_d \subseteq b^{k_2} \left( a_{\alf, \betf} \left( b g(T), T \right) \Lambda_{\Theta} \cap \mathcal{B}_d \right) = \{ {\bf 0} \}.
    $$

    \item Suppose the trajectory is bounded, and $0<l<1$ is fixed so that for any $T \gg 1$ one has 

    $$
    a_{\alf, \betf} \left(  g(T), T \right) \Lambda_{\Theta} \cap l \mathcal{B}_d = \{ {\bf 0} \};
    $$

    equivalently, 

    $$
    a_{l\alf, l\betf} \left(  g(T), T \right) \Lambda_{\Theta} \cap \mathcal{B}_d = \{ {\bf 0} \}.
    $$

    By quasimultiplicativity of the functions $\alpha_i \circ g$ and $\beta_j$, the diagonal operator 

    $$
      a_{\alf, \betf} \left(  g\left(l^{-\frac{1}{k_1}} T\right), l^{\frac{1}{k_1}}T \right) \cdot a_{l\alf, l\betf}^{-1} \left(  g(T), T \right)
    $$

    is expanding. Thus, via the change of variable $U = l^{\frac{1}{k_1}}T$ one can see that for $U \gg 1$

    $$
     a_{\alf, \betf} \left(  g\left(l^{-\frac{2}{k_1}} U\right), U \right)\Lambda_{\Theta} \cap \mathcal{B}_d = \{ {\bf 0} \}.
    $$

    Finally, quasimultiplicativativity of $g$ implies that the diagonal operator 

    $$
      a_{\alf, \betf} \left(  l^{\frac{2k_2}{k_1}} g\left( U \right), U \right) \cdot a_{\alf, \betf}^{-1} \left(  g\left(l^{-\frac{2}{k_1}} U\right), U \right)
    $$

    is expanding, by which

     $$
    a_{\alf, \betf} \left(  l^{\frac{2k_2}{k_1}} g\left( U \right), U \right) \Lambda_{\Theta} \cap \mathcal{B}_d = \{ {\bf 0} \}.
    $$

    for $U$ large enough, and thus $\Theta$ is $(\alf, \betf)$-badly approximable.

    \end{itemize}
\end{proof}

Proposition \ref{bad=bounded} immediately implies the analog of the classical fact: the matrices $\Theta$ and $\Theta^{\top}$ are badly approximable or not simultaneously.

\begin{corollary}\label{transpbad}
    Let $\alf$ and $\betf$ be tuples of quasimultiplicative functions. Suppose $\Theta \in {\bf M}_{n,m}$ is $(\alf, \betf)$-badly approximable. Then, $\Theta^{\top}$ is $(\betf, \alf)$-badly approximable.
\end{corollary}

We obtain the following criterion for bad approximability as a corollary of Theorem \ref{Jarnik_general} and Theorem \ref{weighted_impr}:

{

\begin{corollary}\label{general ba criterion}
Let $\alf$ and $\betf$ be tuples of quasimultiplicative functions.

\begin{itemize}
    \item If a matrix $\Theta \in {\bf M}_{n,m}$ is  $(\alf, \betf)$-badly approximable, then there exists $k > 0$ such that any pair $(\Theta, \pmb{\eta})$ is $(k g_{\alf, \betf}, \alf, \betf)$-Dirichlet:

        $$
        {\widehat{\bf D}^{\Theta} [kg_{\alf, \betf},\alf, \betf]} = \mathbb{R}^n.
        $$

       \item If a matrix $\Theta \in {\bf M}_{n,m}$ is not $(\alf, \betf)$-badly approximable, then for any $k > 0$  almost any pair $(\Theta, \pmb{\eta})$ is not $(k g_{\alf, \betf}, \alf, \betf)$-Dirichlet:

        $$
        m \left({\widehat{\bf D}^{\Theta} [kg_{\alf, \betf},\alf, \betf]} \right) = 0.
        $$
\end{itemize}
\end{corollary}

}

\subsection{Bad and very good approximability: approximations with weights}\label{weights_bad_section}

At the end of this section we switch to the special case of approximations with weights (see Section \ref{weightssection} for definitions and notation). We will work with the following classical definitions (the first is a special case of Definition \ref{generalbaddef}):

\begin{definition}
\begin{itemize}
  \item Matrix $\Theta \in {\bf M}_{n,m}$ is called { \it $(\rrho, \ssigma)$-badly approximable} with constant $b$ if it is not $(bf_1; \rrho, \ssigma)$-approximable, and  { \it $(\rrho, \ssigma)$-badly approximable} if it is $(\rrho, \ssigma)$-badly approximable with constant $b$ for some $b > 0$.  

  \item Matrix $\Theta \in {\bf M}_{n,m}$ is called { \it $(\rrho, \ssigma)$-very well approximable} if it is $(f_{1+ \varepsilon}; \rrho, \ssigma)$-approximable for some $\varepsilon > 0$. 
\end{itemize}
    
\end{definition}

This definition of bad approximability aligns with the general definition from Section \ref{bad_general}, since in this case one has $g_{\alf, \betf} = f_1$, as noted in Section \ref{Dirichlet}.

{

We start with slightly strengthening the statement of Corollary \ref{transpbad} by providing an explicit constant; this constant is a direct consequence of Lemma \ref{lem21} \ref{kostya} due to German and Evdokimov, however as far as we know it is not explicitly documented in any literature. 

\begin{proposition}\label{homogeneous bad transference}
    Suppose $\Theta \in {\bf M}_{n,m}$ is $(\rrho, \ssigma)$-badly approximable with constant $b$, and let $r_+ = \max\limits_{i,j} \{ \rho_i, \sigma_j \}$. Let $c$ be as in Lemma \ref{lem21} \ref{kostya}. Then, $\Theta^{\top}$ is $(\ssigma, \rrho)$-badly approximable with constant $\frac{b^{\frac{2}{r_+}-1}}{d}$.
\end{proposition}

We will prove Proposition \ref{homogeneous bad transference} later in this section. Now, we are ready to state a stronger version of Corollary \ref{general ba criterion} for approximations with weights. 

\begin{corollary}\label{bad_weights}

There exist such functions $\kappa(b) = \kappa_{\rrho, \ssigma}(b)$ and $K(b) = K_{\rrho, \ssigma}(b)$ that the following holds:
    \begin{itemize}
        \item If a matrix $\Theta \in {\bf M}_{n,m}$ is  $(\rrho, \ssigma)$-badly approximable with constant $b$, then any pair $(\Theta, \pmb{\eta})$ is $K(b) f_1$-Dirichlet:

        $$
         {\widehat{\bf D}^{\Theta} [K(b) f_{1}; \rrho, \ssigma]} = \mathbb{R}^n.
        $$

        \item If a matrix $\Theta \in {\bf M}_{n,m}$ is not $(\rrho, \ssigma)$-badly approximable with constant $b$, then for any $\kappa < \kappa(b)$  almost any pair $(\Theta, \pmb{\eta})$ is not $\kappa f_1$-Dirichlet:

        $$
        m \left({\widehat{\bf D}^{\Theta} [\kappa f_{1}; \rrho, \ssigma]} \right) = 0.
        $$

        \end{itemize}

        More specifically, let $r_- = \min\limits_{i,j} \{ \rho_i, \sigma_j \}$ and $r_+ = \max\limits_{i,j} \{ \rho_i, \sigma_j \}$; then one can take

        $$
        K(b) = \frac{d^3(d!)^2}{b^{\frac{2}{r_+} - 1}} \,\,\,\,\,\,\,\,\,\,\,\, \text{and} \,\,\,\,\,\,\,\,\,\,\,\, \kappa(b) = 2^{- \frac{2}{r_-}} d^{- \frac{2}{r_-} - \frac{r_+}{2 - r_+}} b^{ - \frac{r_+}{2 - r_+}}.
        $$

\end{corollary}

A statement of a similar form can be shown for very well approximability. Together with Proposition \ref{homogeneous bad transference}, Proposition \ref{DI with weights} and Proposition \ref{lem24} imply the following criterion :

\begin{corollary}\label{VWA_weights}

\begin{itemize}

        \item If a matrix $\Theta \in {\bf M}_{n,m}$ is $(\rrho, \ssigma)$-very well approximable, then there exists $\varepsilon > 0$ such that almost no pair $(\Theta, \pmb{\eta})$ is $[f_{1 - \varepsilon}; \rrho, \ssigma]$-Dirichlet:

        $$
        m \left( {\widehat{\bf D}^{\Theta} [f_{1 - \varepsilon}; \rrho, \ssigma]} \right) = 0.
        $$

        \item If $\Theta$ is not $(\rrho, \ssigma)$-very well approximable, then for any $\varepsilon > 0$ and any $\pmb{\eta} \in \mathbb{R}^n$ the pair $(\Theta, \pmb{\eta})$ is $[f_{1 - \varepsilon}; \rrho, \ssigma]$-Dirichlet:

        $$
        {\widehat{\bf D}^{\Theta} [f_{1 - \varepsilon}; \rrho, \ssigma]}  = \mathbb{R}^n.
        $$
    \end{itemize}
\end{corollary}

We note that statements similar to Corollary \ref{bad_weights} and Corollary \ref{VWA_weights} in terms of Schmidt's games rather than Lebesgue measure in the non-weighted case were recently shown in \cite{keith}: see Theorem 2.1 and Theorem 2.3.

At the end of this section, we provide an argument proving both Proposition \ref{homogeneous bad transference} and Corollary \ref{bad_weights}. The tools and ideas of this proof are the same as the ones used in Proposition \ref{bad=bounded} and Theorem \ref{Jarnik_general}; however, it is more convenient to trace the particular constants using Observation \ref{logarithmic criterion} instead of Observation \ref{geom_sets}.

\begin{proof}
    By Observation \ref{logarithmic criterion}, $\Theta$ is $(\rrho, \ssigma)$-badly approximable with constant $b$ if and only if $a_{\rho, \sigma}(T) \Lambda_{\Theta} \cap P_{\rrho, \sigma}(b) = \{ {\bf 0} \}$ for any $T$ large enough.

    \vskip+0.3cm

    { \bf We start with the assumption that $\Theta$ is $(\rrho, \ssigma)$-badly approximable with constant $b$}, which implies 

    $$
    a_{\rho, \sigma}(t) \Lambda_{\Theta} \cap c \cdot \frac{1}{c} P_{\rrho, \sigma}(z) = \{ {\bf 0} \}
    $$ 

    and therefore by Lemma \ref{lem21} \ref{kostya} and Observation \ref{obs_dual}

    \begin{equation}\label{intermediate transference}
    a_{\rho, \sigma}\left(\frac{1}{T}\right) \Lambda_{-\Theta}^{\top} \cap \left( \frac{1}{c} P_{\rrho, \sigma}(b)\right)^* = a_{\rho, \sigma}\left(\frac{1}{T} \right) \Lambda_{-\Theta}^{\top} \cap \frac{1}{c^{d-1}} P_{\rrho, \sigma}(b)^* = \{ {\bf 0} \},
    \end{equation}

    where $c$ is as in Lemma \ref{lem21} \ref{kostya}. A direct calculation shows that 

    $$
    w = \frac{b^{\frac{2}{r_+}-1}}{c^{2(d-1)}}
    $$

    is the smallest number with the condition $P_{\rrho, \sigma}(w) \subseteq  \frac{1}{c^{d-1}} P_{\rrho, \sigma}(b)^*$, guaranteeing that 

    $$
    a_{\rho, \sigma}\left( \frac{1}{T} \right) \Lambda_{-\Theta}^{\top} \cap P_{\rrho, \sigma}(w) = \{ {\bf 0} \}
    $$

    and thus that $\Theta^{\top}$ is $(\ssigma, \rrho)$-badly approximable with constant $w = \frac{b^{\frac{2}{r_+}-1}}{c^{2(d-1)}} = \frac{b^{\frac{2}{r_+}-1}}{d}$. { \bf This concludes the proof of Proposition \ref{homogeneous bad transference}. }

    \vskip+0.3cm

    By Lemma \ref{lem21} \ref{mahler},  \eqref{intermediate transference} implies that for any $T$ large enough, for any $\pmb{\gamma} \in \mathbb{R}^d$

    $$
    a_{\rho, \sigma}(T) \Lambda_{\Theta} \cap \left( \frac{C \cdot c^{d-1}}{b}P_{\rrho, \sigma}(b) + \pmb{\gamma} \right) = \emptyset
    $$
    and in addition contains nonzero points if $\pmb{\gamma} = {\bf 0}$. In turn, this implies that for any $T$ large enough and for any $\pmb{\eta} \in \mathbb{R}^n$ the set $a_{\rho, \sigma}(T) \Lambda_{\Theta, \pmb{\eta}} \cap \frac{C \cdot c^{d-1}}{b}P_{\rrho, \sigma}(b)$ is nonempty, and contains nonzero points if $\pmb{\eta} = {\bf 0}$. It remains to find the minimal number $u$ for which $\frac{C \cdot c^{d-1}}{b}P_{\rrho, \sigma}(b) \subseteq P_{\rrho, \sigma}(u)$ and use Observation \ref{logarithmic criterion} to conclude that any pair $(\Theta, \pmb{\eta})$ is $(Kf_1, \rrho, \ssigma)$-Dirichlet with $K \geq K(b) : = u = \frac{dC^2}{b^{\frac{2}{r_+} - 1}}$, { \bf which concludes the proof of the first part of Corollary \ref{bad_weights}.}

    \vskip+0.3cm

    Finally, to show the second part of Corollary \ref{bad_weights} we assume that $\Theta$ is not  $(\rrho, \ssigma)$-badly approximable with constant $b$. By Proposition \ref{homogeneous bad transference}, it means that $\Theta^{\top}$ is not $(\ssigma, \rrho)$-Badly approximable with constant $k = (bd)^{\frac{r_+}{2-r_+}}$. By Proposition \ref{DI with weights}, for any $\varepsilon < (2d)^{-\frac{1}{r_-}}$ the set 

    $$
    \widehat{\bf D}_{n,m}^{\Theta} \left[ \varepsilon^2 k f_1; \rrho, \ssigma  \right]
    $$

has Lebesgue measure zero, which finishes the proof as $\kappa = \varepsilon^2 k \leq \kappa(b)$.

\end{proof}

\color{red}

}

\section{Proofs}\label{section proofs}

\subsection{Proof of Theorem \ref{Jarnik_general}}\label{sect2}
    
    The proof is similar to the proof of Lemma 2.4 from \cite{CGGMS20}. By assumption, the matrix $\Theta^{\top}$ is not $(\uf; \betf, \alf)$-approximable; by Observation \ref{geom_sets}, it means that 

    \begin{equation}\label{largeenough}
    a_{\alf, \betf} \left( U, f(U) \right) \Lambda_{-\Theta}^{\top} \cap \mathcal{B}_d = \{ {\bf 0} \}
    \end{equation}

for any $U$ large enough. Recall that $a_{\alf, \betf} \left( \frac{1}{U}, \frac{1}{f(U)} \right)  \Lambda_{\Theta}$ is dual to $ a_{\alf, \betf} \left( U, f(U) \right) \Lambda_{-\Theta}^{\top}$ (Observation \ref{obs_dual}) and the cube $\mathcal{B}_d$ is self-pseudo-compound (Observation \ref{obs_cube}), so by Lemma \ref{lem21} \ref{mahler} for any $\pmb{\gamma} \in \mathbb{R}^d$ there exists such a $U_0$ that for $U > U_0$ the set 

$$
    \left(C \mathcal{B}_d + \pmb{\gamma} \right) \cap a_{\alf, \betf} \left( \frac{1}{U}, \frac{1}{f(U)} \right) \Lambda_{\Theta}
    $$

    is nonempty, and contains nonzero points if $\pmb{\gamma} = 0$. Next, let $T = \frac{1}{f(U)}$ and thus $U = \frac{1}{g(T)}$; then, there exists such a $T_0$ that for $T > T_0$ the set 

    \begin{equation}\label{nonempty set}
        \left(C \mathcal{B}_d + \pmb{\gamma} \right) \cap a_{\alf, \betf} \left( g(T), T \right) \Lambda_{\Theta} = 
    \end{equation}

    is nonempty, and contains nonzero points if $\pmb{\gamma} = 0$. Now, let $\alf' = C \alf$ and $\betf' = C \betf$. Fix $\pmb{\eta} \in \mathbb{R}^n$ and consider $\pmb{\gamma} = C \cdot a_{\alf', \betf'} \left( g(T), T \right) \begin{pmatrix}
        \pmb{\eta} \\ {\bf 0}
    \end{pmatrix} $. A direct calculation shows that if the set \eqref{nonempty set} is nonempty, then the set

    \begin{equation}\label{nonempty set 2}
    \left(  a_{\alf', \betf'} \left( g(T), T \right) \Lambda_{\Theta, \pmb{\eta}} \right) \cap \mathcal{B}_d
    \end{equation}

    is also nonempty, which by Observation \ref{geom_sets} means for nonzero $\pmb{\eta}$ that $(\Theta, \pmb{\eta}) \in \widehat{\bf D}_{n,m}\left[\ug; C \alf, C \betf\right]$. In the case $\pmb{\eta} = {\bf 0}$ it remains to notice that the set \eqref{nonempty set 2} contains a nonzero point.

\subsection{Proof of Theorem \ref{weighted_impr} and Proposition \ref{DI with weights}}

The proof of Theorem \ref{weighted_impr} will be based on two lemmata.

Suppose the assumptions of Theorem \ref{weighted_impr} are satisfied: $\Theta^T \in {\bf W}_{m,n}[\uf; \betf, \alf]$. Let ${ \bf y}_{\nu}$ be a sequence of vectors in $\mathbb{Z}^n$ for which 
\begin{equation}\label{ynudef}
    || \Theta^T { \bf y}_\nu ||_{\betf} \leq f(|{ \bf y}_\nu|_{\alf});
\end{equation}

we will use the notation $Y_\nu = |{ \bf y}_\nu|_{\alpha}$.

By quasimuiltiplicativity, for any $\delta > 0$ there exists $\varepsilon = \varepsilon(\delta) > 0$ such that 
\begin{equation}\label{eps(delta)} 
\alpha_i(T) \alpha_i\left( \frac{\varepsilon}{T} \right) \leq \delta \,\,\,\,\,\,\,\,\,\,\,\, \text{and} \,\,\,\,\,\,\,\,\,\,\,\, \beta_j(\varepsilon T) \beta_j\left( \frac{1}{T} \right) \leq \delta  \,\,\,\,\,\,\,\,\,\,\,\, \text{for any} \,\, i ,j  \,\, \text{and any} \,\, T \in \mathbb{R}_+.
\end{equation}

Let $\tilde{g} = \tilde{g}_{\varepsilon(\delta)} = \varepsilon \cdot \ug \circ \frac{1}{\varepsilon}$.

\vskip+0.2cm

    We define 

    $$
    \Omega_{\nu} = \Omega_{\nu}\left((m+n)\delta\right): = \{ \pmb{\eta} \in [0,1)^n: \,\,\, ||\pmb{\eta} \cdot { \bf y}_{\nu}|| \leq (m + n) \delta \}.
    $$

\begin{lemma}\label{transferencepart}

Let $\delta > 0$ be a parameter, and $\y_{\nu}, \tilde{g}$ and $\Omega_{\nu}$ be chosen as above. Then,

    $$
    \widehat{\bf D}_{n,m}^{\Theta}[\tilde{g}; \alf, \betf] \cap [0, 1)^n \subseteq \bigcup\limits_{N=1}^{\infty} \bigcap\limits_{\nu = N}^{\infty} \Omega_{\nu}.
    $$
\end{lemma}

\begin{proof}
    Suppose $\pmb{\eta} \in \widehat{\bf D}_{n,m}^{\Theta}[\tilde{g}; \alf, \betf] \cap [0, 1)^n$, and let $T_{\nu} = \frac{\varepsilon}{f(Y_\nu)}$; for $\nu$ large enough there exists ${\bf q}_{\nu}$ with $|{ \bf q}_{\nu}|_{\betf} \leq T_{\nu}$ and $ || \Theta  {\bf q}_{\nu} - \pmb{\eta} ||_{\alf} \leq \tilde{g}(T_{\nu})$.

    We will base our proof on a classical transference observation: one has

    \begin{equation}\label{transference}
     \pmb{\eta} \cdot { \bf y}_{\nu} = \sum\limits_{j=1}^m q_j \Theta_j^T({ \bf y}_{\nu}) - \sum\limits_{i=1}^n \left(  \Theta_i({\bf q}) - \eta_i \right) y_{\nu, i},
    \end{equation}

    and since

    $$
    || q_j \Theta_j^T({ \bf y}_{\nu})|| \leq |q_i|\cdot||\Theta_j^T({ \bf y}_{\nu})|| \leq \beta_j(T_{\nu}) \beta_j(f(Y_\nu)) = \beta_j \left(\frac{\varepsilon}{f(Y_\nu)} \right) \beta_j(f(Y_\nu))
    $$

    
    and 

    $$
    || \left(  \Theta_i({\bf q}) - \eta_i \right) y_{\nu, i}|| \leq | y_{\nu, i}|\cdot||\left(  \Theta_i({\bf q}) - \eta_i \right)|| \leq \alpha_i(Y_\nu) \alpha_i(\tilde{g}(T_\nu)) = \alpha_i(Y_\nu) \alpha_i \left( \frac{\varepsilon}{Y_{\nu}} \right),
    $$

    by the choice of $\varepsilon$ we get that 

    \begin{equation}\label{prime}
    ||\pmb{\eta} \cdot { \bf y}_{\nu}|| \leq \sum\limits_{j=1}^m || q_j \Theta_j^T({ \bf y}_{\nu})|| + \sum\limits_{i=1}^n || \left(  \Theta_i({\bf q}) - \eta_i \right) y_{\nu, i}|| \leq (m + n) \delta.
    \end{equation}

    The condition \eqref{prime} means that $\pmb{\eta} \in \Omega_{\nu}$ for $\nu$ large enough, that is, $\widehat{\bf D}_{n,m}^{\Theta}[\tilde{g}; \alf, \betf] \cap [0, 1)^n \subseteq \bigcup\limits_{N=1}^{\infty} \bigcap\limits_{\nu = N}^{\infty} \Omega_{\nu}.$
\end{proof}

\vskip+0.3cm

We want to show that for the correct choice of parameters the liminf set from Lemma \ref{transferencepart} is a nullset:

\begin{lemma}\label{measurepart}
    Let $0 < \delta < \frac{1}{2(n+m)}$, and $\y_{\nu}$ be a sequence of integer vectors satisfying \ref{ynudef} with an additional condition $\lim\limits_{\nu \rightarrow \infty} \frac{|\y_{\nu + 1}|}{|\y_{\nu}|} = \infty$. Then, 

    $$
    m \left( \bigcup\limits_{N=1}^{\infty} \bigcap\limits_{\nu = N}^{\infty} \Omega_{\nu} \right) = 0.
    $$
\end{lemma}

\begin{proof}
    In this proof, we will work with sets of the form

    $$
    E_{ \y, a } (\Delta): = \{ \x \in \mathbb{R}^n: \,\,\, || \x \cdot \y + a || < \Delta \}, 
    $$

    where $\y \in \mathbb{R}^n, a \in \mathbb{R}$ are parameters, and $0 < \Delta < \frac{1}{2}$. The following three observations can be directly checked:

    \begin{obs}\label{O1.}
       $m \left( E_{\y, a} (\Delta) \right) \cap [0, 1)^n = 2 \Delta \left( 1 + \xi(|\y|) \right)$, where $\xi(|\y|)$ is infinitely small as $|\y| \rightarrow \infty$; 
    \end{obs}
        
        \begin{obs}\label{O2.}
         Let $B: = \prod\limits_{i=1}^n [b_i \cdot 2^{-k}; (b_i+1) \cdot 2^{-k})$ be a cube, and $\bfb = (b_1, \ldots, b_n)$. We note that under a transformation $\x  \mapsto 2^k \x + \bfb$ one has

        $$
        B \mapsto [0, 1)^n \,\,\,\,\,\,\,\,\,\,\,\, \text{and} \,\,\,\,\,\,\,\,\,\,\,\, E_{ \y, a } (\Delta) \mapsto E_{ \frac{\y}{2^k}, a' } (\Delta)
        $$

        where $a' = a - \frac{\bfb \cdot \y}{2^k}$. Therefore, by {Observation } \ref{O1.}, 

        $$
        m \left( E_{\y, a} (\Delta) \right) \cap B = m(B) \cdot 2 \Delta \left( 1 + \xi \left(\frac{|\y|}{2^k} \right) \right).
        $$
    \end{obs}
   
        \begin{obs}\label{O3.}
         Let $Q_k$ be the union of all cubes of the form 
        $$\prod\limits_{i=1}^n [b_i \cdot 2^{-k}; (b_i+1) \cdot 2^{-k}), \,\,\, b_i \in \mathbb{Z},
        $$ intersecting $E_{\y, a} (\Delta)$. Then,

        $$
        \frac{m \left( Q_k \right)}{m \left( E_{ \y, a } (\Delta) \right)} \leq 1 + \zeta_{\Delta}\left(\frac{2^k}{|\y|}\right), 
        $$

        where $\zeta_{\Delta}\left( z \right)$ is an infinitely small function as $z \rightarrow \infty$.
    \end{obs}

    Let us note that

    $$
    \Omega_{\nu}^c: = [0,1)^n \setminus \Omega_{\nu} = \{ \pmb{\eta} \in [0,1)^n: \,\,\, |\pmb{\eta} \cdot { \bf y}_{\nu} - p + \frac{1}{2}| < \frac{1}{2} -  (m + n) \delta \,\,\, \text{for some} \,\, p \in \mathbb{Z} \} = E_{ \y_{\nu}, \frac{1}{2} } (\Delta) \cap [0, 1)^n,
    $$

  where $\Delta = \frac{1}{2} -  (m + n) \delta > 0$.
    
   Fix $\lambda > 0$. Our goal is to show the following

    \begin{indentedclaim}
    For $\nu \neq \mu$ large enough, 
    $$
    m \left( \Omega_{\nu}^c \cap \Omega_{\mu}^c \right) \leq (1 + \lambda)  m \left( \Omega_{\nu}^c \right) m \left(\Omega_{\mu}^c \right) .
    $$

    \vskip+0.3cm

    Indeed: let $Q_k$ be as in Observation \ref{O3.} for $\Omega_{\nu}^c = E_{ \y_{\nu}, \frac{1}{2} } (\Delta) \cap [0, 1)^n$. Then,

    $$
    m \left( \Omega_{\nu}^c \cap \Omega_{\mu}^c \right) \leq m \left( Q_k \cap \Omega_{\mu}^c \right) \stackunder[1pt]{{}={}}{\scriptstyle \text{Observation } \ref{O2.}} m(Q_k)  \cdot 2 \Delta \left( 1 + \xi \left(\frac{|\y_\mu|}{2^k}  \right)\right) \stackunder[1pt]{{}={}}{\scriptstyle  \text{Observation } \ref{O1.}}
    $$
    $$
    m(Q_k) m(\Omega_{\mu}^c) \cdot \left( 1 + \xi \left({|\y_\mu|}  \right) \right)^{-1} \left( 1 + \xi \left(\frac{|\y_\mu|}{2^k}  \right)\right) \stackunder[1pt]{{}\leq{}}{\scriptstyle  \text{Observation } \ref{O3.}} 
    $$
    $$
    m(\Omega_{\nu}^c) m(\Omega_{\mu}^c) \cdot \left( 1 + \zeta_{\Delta}\left(\frac{2^k}{|\y_{\nu}|}\right) \right) \left( 1 + \xi \left({|\y_\mu|}  \right) \right)^{-1} \left( 1 + \xi \left(\frac{|\y_\mu|}{2^k}  \right)\right);
    $$

    if $\frac{|\y_{\mu}|}{|\y_{\nu}|}$ is large enough, we can choose $k$ in such a way that 

    $$
    \left( 1 + \zeta_{\Delta}\left(\frac{2^k}{|\y_{\nu}|}\right) \right) \left( 1 + \xi \left({|\y_\mu|}  \right) \right)^{-1} \left( 1 + \xi \left(\frac{|\y_\mu|}{2^k}  \right)\right) \leq (1 + \lambda).
    $$
\end{indentedclaim}

\vskip+0.3cm

 By Observation \ref{O1.} again, we see that

    $$
    \sum\limits_{\nu = 1}^{\infty} m \left( \Omega_{\nu}^c \right) = \infty,
    $$

    and we can apply the classical metrical lemma due to Sprindzhuk (Lemma 5, §3, Ch.1 in \cite{Spr}):

    $$
        m \left(  \bigcap\limits_{N=1}^{\infty} \bigcup\limits_{\nu=N}^{\infty} {\Omega}_{\nu}^c \right) \geq \limsup\limits_{N \rightarrow \infty} \frac{\left( \sum\limits_{\nu=1}^{N} m({\Omega}_{\nu}^c) \right)^2}{\sum\limits_{\nu,\mu=1}^{N} m({\Omega}_{\nu}^c \cap {\Omega}_{\mu}^c)} \geq \frac{1}{1 + \lambda}.
        $$

    Note that $\lambda$ can be taken arbitrarily small; thus, 
    $$m \left(  \bigcap\limits_{N=1}^{\infty} \bigcup\limits_{\nu=N}^{\infty} {\Omega}_{\nu}^c \right) = 1$$
    and therefore 
    $$m \left(  \bigcup\limits_{N=1}^{\infty} \bigcap\limits_{\nu=N}^{\infty} {\Omega}_{\nu} \right) = 0.$$
\end{proof}

\vskip+0.3cm

{ \bf Proof of Theorem \ref{weighted_impr}).} Let $0 < \delta < \frac{1}{2(n+m)}$, and $\tilde{g} = \tilde{g}_{\varepsilon(\delta)}$. Let $\{ \y_{\nu} \}$ be a sequence satisfying the conditions of Lemma \ref{measurepart}. Now Theorem \ref{weighted_impr} directly follows from Lemma \ref{transferencepart} and Lemma \ref{measurepart}.

{ \bf Proof of Theorem \ref{DI with weights}).} The existence of some $\varepsilon > 0$ satisfying Proposition \ref{DI with weights} directly follows from Theorem \ref{weighted_impr} applied to the case $\alpha_i(T) = T^{\rho_i}, \, \beta_j(T) = T^{\sigma_j}$. The particular bound on $\varepsilon$ follows if we notice that for our special case one can take $\varepsilon(\delta) = \delta^{\frac{1}{r_-}}$ in \eqref{eps(delta)}, and that it is enough to require the condition $\delta < \frac{1}{2d}$ in order for Lemma \ref{measurepart} to hold.

\subsection{Proof of Theorem \ref{main for pairs}}\label{sect3}

For convenience, we will again use the notation

$$\alpha(T): = \prod\limits_{i=1}^{n} \alpha_i(T) \,\,\,\,\,\,\, \text{and} \,\,\,\,\,\,\,\,\beta(T): = \prod\limits_{j=1}^{m} \beta_j(T).$$ 

Let us first assume that $g$ is a continuous and strictly decreasing function, and prove that the theorem holds with this additional assumption. We need to prove a technical lemma. Note that for the case of approximation with weights, and thus for Proposition \ref{kwad_weighted}, the statement of this lemma becomes a trivial observation.

\begin{lemma}\label{series equivalent}
    Suppose $f$ and $g$ are connected by \eqref{fgrelation}, and all the functions $\alpha_i, \beta_j$ are quasimultiplicative. Then,
    
    \begin{equation}\label{equivalence}
    \sum\limits_{k=1}^{\infty} k^{-1} \beta(f(k)) \cdot \alpha(k) < \infty \,\,\,\,\,\,\,\, \text{if and only if} \,\,\,\,\,\,\,\, \sum\limits_{l=1}^{\infty} \frac{1}{l \cdot \beta(l) \cdot \alpha(g(l))} < \infty.
    \end{equation}
\end{lemma}

\begin{proof}

Note that the functions $\alpha(T)$ and $\beta(T)$ are also quasimultiplicative (Observation \ref{quasiproperties}), and let $c_1, c_2$ and $M$ be the constants from the definition of quasimultiplicativity which work for both $\alpha$ and $\beta$.

{\bf Step 1: Linearizing the functions.} We will introduce the functions $\widehat{\alpha}$ and $\widehat{\beta}$, defined by 
    \begin{equation}\label{star definition}
    \widehat{\alpha}(T) = \frac{M^{k+1} - T}{M^{k+1} - M^k} \alpha(M^k) + \frac{T - M^k}{M^{k+1} - M^k} \alpha (M^{k+1}) \,\,\,\,\,\, \text{for} \,\, M^k \leq T \leq M^{k+1},
    \end{equation}
    $$
    \widehat{\beta}(T) = \frac{M^{k+1} - T}{M^{k+1} - M^k} \beta(M^k) + \frac{T - M^k}{M^{k+1} - M^k} \beta(M^{k+1}) \,\,\,\,\,\, \text{for} \,\, M^k \leq T \leq M^{k+1},
    $$

    that is, piecewise linear functions which are equal to $\alpha$ or $\beta$ at powers of $M$. By quasimultiplicativity of $\alpha$ and $\beta$, for large enough or small enough values of $T$ one has

    $$
    \frac{1}{c_2} \widehat{\alpha}(T) \leq \alpha(T) \leq c_2 \widehat{\alpha}(T) \,\,\,\,\,\,\,\, \text{and} \,\,\,\,\,\,\,\, \frac{1}{c_2} \widehat{\beta}(T) \leq \beta(T) \leq c_2 \widehat{\beta}(T),
    $$

    which immediately implies that the first series in \eqref{equivalence} converges without stars if and only if it converges with stars, and same holds for the second series. Thus it is enough to prove Lemma \ref{series equivalent} for $\alpha = \widehat{\alpha}$ and $\beta = \widehat{\beta}$.

    \vskip+0.3cm

    {\bf Step 2: One property of hat functions.}

    For the remainder of the proof, let us assume that $\alpha = \widehat{\alpha}$ and $\beta = \widehat{\beta}$. We will prove that there exist such positive absolute constants $\lambda_1, \lambda_2$ (which depend on $c_1, c_2$ and $M$ only) that 

    \begin{equation}\label{derivative behavior}
        \lambda_1 \frac{\alpha(t)}{t} \leq \alpha'(t) \leq \lambda_2 \frac{\alpha(t)}{t} \,\,\,\,\,\,\,\,\,\,\,\, \text{and} \,\,\,\,\,\,\,\,\,\,\,\, \lambda_1 \frac{\beta(t)}{t} \leq \beta'(t) \leq \lambda_2 \frac{\beta(t)}{t}
    \end{equation}

for any $t$ large enough and small enough (except the powers of $M$, where $\alpha'$ and $\beta'$ are not defined).

We will show \eqref{derivative behavior} for $\alpha$ and large enough $t$; the proof for small values of $t$ and for the function $\beta$ is similar.

By quasimultiplicativity, for any $k$ large enough one has 

$$
c_1 \alpha(M^k) \leq \alpha(M^{k+1}) \leq c_2 \alpha(M^k);
$$

let us consider $t \in (M^k; M^{k+1})$, and let $\xi: = \frac{\alpha(M^{k+1})}{\alpha(M^k)} \in [c_1; c_2]$. Then, by \eqref{star definition}, 

$$
\alpha'(t) = \frac{\alpha(M^k)}{(M-1)M^k} (\xi - 1)
$$

and 

$$
\frac{\alpha(t)}{t} = \frac{\alpha(M^k)}{(M-1)M^k} \left( (\xi - 1)  + \frac{1}{t} (M^{k+1} - \xi M^k ) \right) = \alpha'(t) \cdot \left( 1 + \frac{M^k}{t} \frac{M - \xi}{\xi - 1} \right),
$$

where for $t > M^k$ one has $\frac{M-1}{c_2 - 1} < 1 + \frac{M^k}{t} \frac{M - \xi}{\xi - 1} < \frac{M-1}{c_1 - 1}$. Taking $\lambda_1 = \frac{c_1-1}{M-1}$ and $\lambda_2 = \frac{c_2 - 1}{M-1}$ provides the desired statement.

\vskip+0.3cm

We will also use one simple corollary of these inequalities: if we set $A(t) = \int\limits_0^{t} \frac{\alpha(s)}{s} ds$, then

\begin{equation}
    \frac{1}{2} \lambda_1 A(t) \leq \alpha(t) \leq 2 \lambda_2 A(t).
\end{equation}

for $t$ large enough.

    \vskip+0.3cm

    {\bf Step 3: Integral equivalence.}

    Let us use the variable change $t = \frac{1}{g(z)} = f^{-1} \left( \frac{1}{z} \right)$, equivalently, $z = \frac{1}{f(t)}$, to show that the integrals

    {\centering\compress
\begin{tabularx}{\linewidth}{>{\leqnomode}XX}
\begin{equation}\label{int1}
       \int\limits_1^{+\infty} \frac{dz}{z \cdot \beta(z) \cdot  \alpha(g(z))} \,\,\,\,\,\,\,\,\,\,\,\,\,\,\,\,\,\,\,\,\,\,\,\,\,\,\text{and}
\end{equation}
 &
 \begin{equation}\label{int2}
\int\limits_1^{+\infty} t^{-1} \beta(f(t)) \cdot \alpha(t) dt
\end{equation}
\end{tabularx} \vspace{-\baselineskip}}

    either both converge or both diverge; together with Step 1, this will complete the proof of the lemma. It will be convenient for us to use the symbol "$\asymp$" for equivalence up to an additive constant (we will use it to simply write integrals "from 1 to infinity" instead of "from $t_0$ to infinity") or up to a multiplucative absolute constant (depending on $\lambda_1$ or $\lambda_2$ only).

    One has:

    $$
    \int\limits_1^{+\infty} t^{-1}  \beta(f(t)) \cdot \alpha(t) dt = \sum\limits_{k=0}^{\infty} \int\limits_{M^k}^{M^{k+1}} \beta(f(t)) \cdot \frac{ \alpha(t)}{t} dt = 
    $$
    $$
    \sum\limits_{k=0}^{\infty} \left( A(t) \beta(f(t))\rvert_{t = M^k}^{t = M^{k+1}} - \int\limits_{t = M^k}^{t = M^{k+1}} A(t) \cdot \beta'(f(t)) d f(t)
 \right) \asymp
    $$
    $$
    \sum\limits_{k=0}^{\infty} \left( A(t) \beta(f(t))\rvert_{t = M^k}^{t = M^{k+1}} - \int\limits_{z = \frac{1}{f(M^{k})}}^{z = \frac{1}{f(M^{k+1})}} A\left(\frac{1}{g(z)}\right) \cdot \beta'\left(\frac{1}{z} \right) d \left( \frac{1}{z} \right)
 \right) =
    $$
    $$
    A(t) \beta(f(t))\rvert_{t = 1}^{t = + \infty} + \sum\limits_{k=0}^{\infty} \int\limits_{z = \frac{1}{f(M^{k})}}^{z = \frac{1}{f(M^{k+1})}} A\left(\frac{1}{g(z)}\right) \cdot \beta'\left(\frac{1}{z} \right) \frac{1}{z^2} d z \asymp
    $$
    $$
    \lim\limits_{t \rightarrow \infty } \beta(f(t)) \cdot \alpha(t) + \sum\limits_{k=0}^{\infty} \int\limits_{z = \frac{1}{f(M^{k})}}^{z = \frac{1}{f(M^{k+1})}} \alpha \left( \frac{1}{g(z)} \right) \cdot \beta \left(\frac{1}{z} \right)  \cdot \frac{z}{z^2} dz \asymp
    $$
    $$
    \lim\limits_{t \rightarrow \infty } \beta(f(t)) \cdot \alpha(t) + \int\limits_{1}^{\infty} \frac{dz}{z \cdot \beta(z) \cdot  \alpha(g(z))},
    $$

 where both terms are nonnegative.

    Thus, if the integral \eqref{int2} converges, then \eqref{int1} also converges.   Conversely: suppose \eqref{int1} converges. Let us notice that 
    
    $$
    \frac{1}{\beta(z) \cdot  \alpha(g(z))} = \beta(f(t)) \cdot \alpha(t),
    $$

    where $t$ and $z$ are related as above. Therefore, the convergence of \eqref{int1} implies that $\lim\limits_{t \rightarrow \infty} \beta(f(t)) \alpha(t) = 0$, and thus \eqref{int2} also has to converge.

    \vskip+0.3cm
    
This observation concludes the proof of Step 3 and therefore of lemma.

\end{proof}

    Let us also notice that multiplication of $f$ or $g$ by any constant and linear variable change (due to quasimultiplicativity) do not change the convergence properties of the corresponding series. Fix $\varepsilon$ small enough, so that Theorem \ref{weighted_impr} holds. We will now prove Theorem \ref{main for pairs} for continuous and strictly decreasing functions $g$.

    \begin{itemize}
        \item Suppose \eqref{kwseries} diverges; by Lemma \ref{series equivalent} it implies that \eqref{khgr series} also diverges. Then, by Theorem \ref{general_Khgr}, for almost every $\Theta \in M_{n,m}$, its transposed matrix $\Theta^{\top} \in M_{m,n}$ belongs to $\bigcap\limits_{\delta > 0 }\bigcap\limits_{C > 0 }{ \bf W}_{m, n}\left[\delta f (C T); \betf, \alf \right]$. Therefore, by Theorem \ref{weighted_impr}, for almost every $\Theta \in M_{n,m}$ one has 
        $$
        m \left( \bigcup\limits_{\delta > 0 }\bigcup\limits_{C > 0 } \widehat{ \bf D}_{n,m}^{\Theta} \left[ C\tilde{g}(\delta T); \alf, \betf \right] \right) = 0,
        $$

        where $C \tilde{g}(\delta T)  = C \varepsilon \cdot g \left( \delta \cdot \frac{1}{\varepsilon}  T \right)$. Taking 

        $$
        \delta = \varepsilon \,\,\,\,\,\,\,\,\,\,\,\, \text{and} \,\,\,\,\,\,\,\,\,\,\,\, C =\frac{1}{\varepsilon},
        $$

        we see that $C \tilde{g}(\delta T) = g(T)$ and deduce that for almost every $\Theta \in M_{n,m}$, 
        $$
        m \left( \widehat{ \bf D}_{n,m}^{\Theta} \left[ g; \alf, \betf \right]  \right) = 0.
        $$

        The set $\widehat{ \bf D}_{n,m} \left[ g; \alf, \betf \right]$ is Borel and thus measurable. Fubini's theorem applied to the characteristic function of this set shows that

        $$
        m \left( \widehat{ \bf D}_{n,m} \left[ g; \alf, \betf  \right]  \right) = 0.
        $$

        \item Suppose \eqref{kwseries} converges; by Lemma \ref{series equivalent} then \eqref{khgr series} also converges. By Theorem \ref{general_Khgr}, we have 
        $$
        m \left( { \bf W}_{m, n}\left[f; \frac{1}{C} \betf, \frac{1}{C}\alf \right] \right) = 0, 
        $$

        where $C$ is as in Theorem \ref{Jarnik_general}.
        By Theorem \ref{Jarnik_general}, for almost every $\Theta \in M_{n,m}$

        $$
        \widehat{\bf D}_{n,m}^{\Theta} \left[ g; \alf, \betf \right] = \mathbb{R}^n.
        $$

        Applying Fubini's Theorem again completes the proof.
    \end{itemize}

    \vskip+0.3cm

    It remains to show that the statement of theorem holds for all non-increasing functions $g$, without continuity or strict monotonicity assumptions. We will use two observations to achieve that.

    \begin{obs}\label{remove_continuity}
        Let $g: \,\,\, \mathbb{R}_+ \rightarrow \mathbb{R}_+$ be a non-increasing function. There exists such a continuous, non-increasing function $h: \,\,\, \mathbb{R}_+ \rightarrow \mathbb{R}_+$ that

        \begin{equation}\label{converge_together}
        \sum\limits_{l=1}^{\infty} \frac{1}{l \cdot \beta(l) \cdot \alpha(g(l))} < \infty \,\,\,\,\,\,\,\, \text{if and only if} \,\,\,\,\,\,\,\, \sum\limits_{l=1}^{\infty} \frac{1}{l \cdot \beta(l) \cdot \alpha(h(l))} < \infty
        \end{equation}

        and 

        \begin{equation}\label{approx_together}
        (\Theta, \pmb{\eta}) \in \widehat{\bf D}_{n,m}[\ug; \alf, \betf]  \,\,\,\,\,\,\,\, \text{if and only if} \,\,\,\,\,\,\,\,  (\Theta, \pmb{\eta}) \in \widehat{\bf D}_{n,m}[h; \alf, \betf].
        \end{equation}
    \end{obs}

    \begin{proof}
        Functions $\beta_j$ are strictly increasing, therefore the set $W: = \mathbb{Z}_+ \cup \bigcup\limits_{j=1}^{m} \beta_j^{-1}(\mathbb{Z}_+)$ is locally finite. By definition of the sets ${\bf D}_{n,m}$, it is enough for the values of $g$ and $h$ to coincide on the set $W$ to guarantee that \eqref{approx_together} and \eqref{converge_together} hold. We define $h(T): = g(T)$ for $T \in W$ and extend it linearly between any two subsequent points of $W$.
    \end{proof}

    \begin{obs}\label{remove_decreasing}
        Let $h: \,\,\, \mathbb{R}_+ \rightarrow \mathbb{R}_+$ be a continuous, non-increasing function. There exist such continuous, strictly decreasing functions $\mathfrak{h}_-$ and $\mathfrak{h}_+$ that 

        \begin{equation}\label{inequality_h}
        \mathfrak{h}_-(T) \leq h(T) \leq \mathfrak{h}_+(T) \,\,\,\,\,\,\,\,\,\,\,\, \text{for any} \,\,\,\,T \in \mathbb{R}_+
        \end{equation}

        and 

    \begin{equation}\label{hconverge_together}
        \sum\limits_{l=1}^{\infty} \frac{1}{l \cdot \beta(l) \cdot \alpha(\mathfrak{h}_-(l))} < \infty \,\,\,\,\,\,\,\, \text{if and only if} \,\,\,\,\,\,\,\, \sum\limits_{l=1}^{\infty} \frac{1}{l \cdot \beta(l) \cdot \alpha(\mathfrak{h}_+(l))} < \infty.
        \end{equation}

    \end{obs}

    \begin{proof}
    It is clear that the statement is true when $H: = \lim\limits_{T \rightarrow \infty} h(T) > 0$: in this case we can take any two continuous strictly decreasing functions with the condition $ \lim\limits_{T \rightarrow \infty} h_-(T) \leq H \leq  \lim\limits_{T \rightarrow \infty} h_+(T)$ (and then adjust on a bounded set if necessary). Assume now that $\lim\limits_{T \rightarrow \infty} h(T) = 0$.
    
        Let $\mathfrak{h}(T)$ be a piecewise linear non-increasing function which values coincide with the values of $h$ on integer values of the argument. Then the functions 

        \begin{equation}\label{shifted_ineq}
        \mathfrak{h}_-(T): = \mathfrak{h}(T+1)  \leq h(T) \leq  \mathfrak{h}_+(T): = \mathfrak{h}(T-1)
        \end{equation}

        satisfy \eqref{inequality_h}. 

        Now consider such positive integers $l$ and $k$ (if they exist) that 

        $$
        h(l-1) > h(l)  = \ldots = h(l+k) > h(l+k+1).
        $$

        We modify the function $\mathfrak{h}_-$ on the segment $[l-1, l+k]$ as follows: we set $\mathfrak{h}_-(l+k-1) : =\frac{1}{2} \left( {h}(l+k) + h(l+k+1) \right)$, do not change the values at endpoints and extend $\mathfrak{h}_-$ linearly on $[l-1, l+k-1]$ and $[l+k-1, l+k]$. This guarantees that $\mathfrak{h}_-(T)$ is strictly decreasing and 
        
        \begin{equation}\label{h-ineq}
        \frac{1}{2} \mathfrak{h}(T+1) \leq \mathfrak{h}_-(T) \leq \mathfrak{h}(T+1)
        \end{equation}
        
        for $T \in [l-1, l+k]$. We apply these procedure to each such segment where $\mathfrak{h}_-$ is constant.

        We apply a similar procedure to $\mathfrak{h}_+$ on the segment $[l, l+k+1]$, setting $\mathfrak{h}_+(l+1) : =\frac{1}{2} \left( {h}(l-1) + h(l) \right)$; we obtain a strictly increasing function $\mathfrak{h}_+$ for which the condition 
        
        \begin{equation}\label{h+ineq}
        \mathfrak{h}(T-1) \leq \mathfrak{h}_+(T) \leq 2\mathfrak{h}(T-1)
        \end{equation}
        
        holds. 

        The conditions \eqref{shifted_ineq}, \eqref{h-ineq} and $\eqref{h+ineq}$ imply \eqref{inequality_h}. The change of variable $l \mapsto l-1$ or $l \mapsto l+1$ in the sums \eqref{hconverge_together} together with the observation that $$\beta(T-1) \asymp \beta(T) \asymp \beta(T+1) \,\,\,\,\,\,\,\, \text{and} \,\,\,\,\,\,\,\, \alpha\left(2h(T)\right) \asymp \alpha\left(h(T)\right) \asymp \alpha\left(\frac{1}{2} h(T)\right)$$
        
        by quasimultiplicativity and the inequalities above implies that both sums in \eqref{hconverge_together} converge or diverge simultaneously with the sum

        $$
         \sum\limits_{l=1}^{\infty} \frac{1}{l \cdot \beta(l) \cdot \alpha({h}(l))}.
        $$
    \end{proof}

    {\bf Proof of Theorem \ref{main for pairs}.} Suppose $g: \,\,\, \mathbb{R}_+ \rightarrow \mathbb{R}_+$ is a non-increasing function, and the series \eqref{kwseries} converges. Construct a continuous function $h$ using Observation \ref{remove_continuity} and then a continuous strictly decreasing function $\mathfrak{h}_-$ based on this $h$ using Observation \ref{remove_decreasing}. The series 

    $$
    \sum\limits_{l=1}^{\infty} \frac{1}{l \cdot \beta(l) \cdot \alpha(\mathfrak{h}_-(l))}
    $$

    converges, and so, as shown above, the set  $\widehat{\bf D}_{n,m}[\mathfrak{h}_-; \alf, \betf]$ has full Lebesgue measure. By Observations \ref{remove_decreasing} and then \ref{remove_continuity}, 
    $$\widehat{\bf D}_{n,m}[\mathfrak{h}_-; \alf, \betf] \subseteq \widehat{\bf D}_{n,m}[h; \alf, \betf]  = \widehat{\bf D}_{n,m}[g; \alf, \betf]
    $$

    and thus the set $\widehat{\bf D}_{n,m}[g; \alf, \betf]$ has full Lebesgue measure.

    \vskip+0.2cm

    In the case when the series \eqref{kwseries} diverges, we apply a similar argument using the function $\mathfrak{h}_+$.

\vskip+0.3cm

\begin{remark}
    Using Lemma \ref{series equivalent}, we can similarly deduce Theorem \ref{general_Khgr} from Theorem \ref{main for pairs} through Theorem \ref{Jarnik_general} and Theorem \ref{weighted_impr}. Indeed: Theorem \ref{main for pairs} tells us whether the measure of almost all sets $\widehat{\bf D}_{n,m}^{\Theta} \left[ g; \alf, \betf \right]$ is full or zero. Knowing this for a fixed $\Theta$ (and choosing the correct constants in front of or inside the functions if necessary), one can use Theorem \ref{Jarnik_general} and Theorem \ref{weighted_impr} to establish whether $\Theta$ belongs to ${ \bf W}_{m, n}\left[f; \betf, \alf \right]$ or not.
\end{remark}

\section*{Acknowledgements}

The author is deeply grateful to Dmitry Kleinbock for insightful discussions and invaluable comments. Sincere thanks are also due to Nikolay Moshchevitin for his generous support, and to Oleg German for kindly providing helpful references.

 \end{document}